\pdfoutput=1                      

\documentclass{amsart}

\usepackage[accsupp]{axessibility}    

\usepackage{amsmath}
\usepackage{amssymb}
\usepackage{amsthm}
\usepackage{color}
\usepackage{xcolor}
\usepackage{verbatim}
\usepackage{tabularx}

\usepackage{mathrsfs}

\AtBeginDocument{\colorlet{defaultcolor}{.}}

\theoremstyle{theorem}
\newtheorem{theorem}{Theorem}

\newtheorem{conjecture}[theorem]{Conjecture}
\newtheorem{corollary}[theorem]{Corollary}

\newtheorem{definition}[theorem]{Definition}

\newtheorem{question}[theorem]{Question}

\newtheorem{lemma}[theorem]{Lemma}

\newtheorem{proposition}[theorem]{Proposition}
\newtheorem{remark}[theorem]{Remark}

\newcommand{\R}{\mathbb{R}}

\newcommand{\metric}{\langle \, , \, \rangle}
\newcommand{\disp}{\displaystyle}
\newcommand{\ra}{\rightarrow}

\newcommand{\eps}{\varepsilon}

\newcommand{\II}{\mathrm{II}}

\newcommand{\di}{\mathrm{d}}

\newcommand{\HH}{\mathbb{H}}

\newcommand{\Ricc}{\mathrm{Ric}}
\newcommand{\cut}{\mathrm{cut}}
\newcommand{\vol}{\mathrm{vol}}
\newcommand{\lip}{\mathrm{Lip}}
\newcommand{\loc}{\mathrm{loc}}

\newcommand{\tcr}{\textcolor{defaultcolor}}
\newcommand{\tcb}{\textcolor{defaultcolor}}
\newcommand{\tcg}{\textcolor{defaultcolor}}

\newcommand{\LL}{\mathscr{L}}
\newcommand{\Sec}{\mathrm{Sec}}
\newcommand{\BB}{\mathbb{B}}

\newcommand{\Ric}{\mathrm{Ric}}
\newcommand{\RR}{\mathbb{R}}
\newcommand{\VV}{\mathscr{V}}
\newcommand{\pmgh}{\mathrm{pmGH}}
\newcommand{\gh}{\mathrm{GH}}
\newcommand{\mes}{\mathfrak{m}}
\newcommand{\ber}{\mathscr{B}}

\newcommand{\RCD}{\mathsf{RCD}}

\DeclareMathOperator{\diver}{div \,}
\DeclareMathOperator{\diverge}{div}
\DeclareMathOperator{\dist}{dist}

\DeclareMathOperator{\tr}{\mathrm{Tr}}

\renewcommand{\div}{\mathrm{div}}

\begin{document}

\title[Splitting for minimal graphs]{Non-negative Ricci curvature and Minimal graphs with linear growth}

\author{Giulio Colombo}
\address{Dipartimento di Matematica ``F. Enriques", Universit\`a degli studi di Milano, Via Saldini 50, I-20133 Milano (Italy).}
\email{giulio.colombo@unimi.it, marco.rigoli55@gmail.com}

\author{Eddygledson S. Gama}
\address{Departamento de Matem\'atica, Universidade Federal de Pernambuco, 50670-901 Recife, Pernambuco (Brazil).}
\email{eddygledson.gama@ufpe.br}

\author{Luciano Mari}
\address{Dipartimento di Matematica ``G. Peano", Universit\`a degli Studi di Torino, Via Carlo Alberto 10, 10123 Torino (Italy).}
\email{luciano.mari@unito.it}

\author{Marco Rigoli}

\thanks{}

\date{\today}

\begin{abstract}
We study minimal graphs with linear growth on complete manifolds $M^m$ with $\Ricc \ge 0$. Under the further assumption that the $(m-2)$-th Ricci curvature in radial direction is bounded below by $C r(x)^{-2}$, we prove that any such graph, if non-constant, forces tangent cones at infinity of $M$ to split off a line. Note that  $M$ is not required to have Euclidean volume growth. We also show that $M$ may not split off any line. Our result parallels that obtained by Cheeger, Colding and Minicozzi for harmonic functions. The core of the paper is a new refinement of Korevaar's gradient estimate for minimal graphs, together with heat equation techniques.    
\end{abstract}



\maketitle
\tableofcontents

\section{Introduction}

The theory of entire minimal graphs in Euclidean space $\R^m$, that is, of functions $u : \R^m \to \R$ solving the minimal (hyper)surface equation
\begin{equation}\label{P}\tag{MSE}
\diver \left( \frac{D u}{\sqrt{1+|D u|^2}}\right) = 0
\end{equation}
is built upon the following foundational results:

\begin{itemize}
\item[$(\ber 1)$] The \emph{Bernstein theorem}: solutions of \eqref{P} are all affine if and only if $m \le 7$.
\item[$(\ber 2)$] For each $m \ge 2$, positive solutions of \eqref{P} are constant.
\item[$(\ber 3)$] For each $m \ge 2$, solutions of \eqref{P} with at most linear growth on one side are affine (i.e., the Hessian $D^2 u \equiv 0$).
\end{itemize}
Here, $u$ is said to have at most linear growth on one side if, up to changing the sign of $u$, 
	\[
	u(x) \ge - C(1+r(x))
	\]
holds on $\R^m$ for some constant $C>0$, where $r$ is the distance from a fixed origin. \\
The validity of $(\ber 1)$ is due, as well-known, to the combined effort of S. Bernstein \cite{bernstein} ($m=2$, see also \cite{emi,hopf_bern}), W.H. Fleming (\cite{fleming}, still for $m=2$), E. De Giorgi (\cite{dg1,dg2}, $m=3$), F. Almgren (\cite{almgren}, $m=4$), J. Simons (\cite{simons}, $m \le 7$) and E. Bombieri, De Giorgi and E. Giusti (\cite{bdgg}, counterexamples if $m \ge 8$). On the other hand, $(\ber 2)$ and $(\ber 3)$ were both proved in \cite{bdgm} by Bombieri, De Giorgi and M. Miranda for $m \ge 3$; in particular, $(\ber 3)$ refines J. Moser's Theorem \cite{jmoser}, which states that $u$ is affine provided that $|Du| \in L^\infty(M)$. Further properties of entire minimal graphs in Euclidean space were obtained in \cite{bg} by Bombieri and Giusti: among them, we mention the fact that $u$ is affine whenever  $m-1$ of its partial derivatives are bounded. The result was improved in recent years by A. Farina \cite{far1,far2}, who showed that $u$ is affine if $m-7$ partial derivatives of $u$ are bounded on one side. Further enhancements  of Moser' result, proving that $D^2u \equiv 0$ by only assuming that $|Du| = o(r)$ as $r(x) \to \infty$, were obtained in \cite{cns,eckerhuisken,Simon}. We also mention the recent \cite{farina_new}, where the rigidity of a minimal graph is obtained by assuming that an upper level set contains, or is contained in, a half-space. \par
In a Riemannian setting, it is natural to ask the following 

\begin{question}
For which classes of complete Riemannian manifolds $M$ one could expect results like $(\ber 1),(\ber 2),(\ber 3)$? 
\end{question}
The problem motivated our previous works \cite{bmpr,cmmr} as well as the present paper. Recall that a solution of \eqref{P} on a Riemannian manifold $(M^m,\sigma)$ gives rise to a graph 
	\[
	F : M \to \R \times M, \qquad F(x) = \big(u(x),x\big) 
	\]
which is minimal if the ambient space $\R \times M$ is endowed with the product metric $\di t^2 + \sigma$. Hereafter, we say that the graph is entire if $u$ is defined on the whole of $M$. \par
If $M$ is close to hyperbolic space $\HH^m$, namely, $M$ is a Cartan-Hadamard manifolds with suitably pinched negative curvature, $(\ber 1),(\ber 2),(\ber 3)$ drastically fail, since each continuous function on the boundary at infinity of $M$ can be attained as the limit value of an entire minimal graph, which is therefore bounded. An exhaustive literature on the problem can be found in the survey \cite{esko_survey}, see also the introduction of \cite{bcmmpr}.\par
Denoting with $g = F^*(\di t^2 + \sigma)$ the graph metric and with $\Delta_g$ its Laplace-Beltrami operator,  equation \eqref{P} can be written as $\Delta_g u = 0$, \tcb{making contact with the theory of harmonic functions. In Euclidean space $M=\R^m$, $(\ber 2)$ and $(\ber 3)$ hold as well when considering harmonic functions instead of solutions to \eqref{P}, while the analogy fails for $(\ber 1)$ since there is no rigidity for entire harmonic functions without imposing any growth condition. This suggests that, for $(\ber 2)$ and $(\ber 3)$, an answer to the above question may be guided by the global behaviour of harmonic functions on Riemannian manifolds, according to which it is natural to consider the problem on manifolds satisfying either}  
%
	\begin{equation}\label{ipo_secric}
	\Sec \ge 0, \qquad \text{or} \qquad \Ricc \ge 0
	\end{equation}
where $\Sec,\Ricc$ are the sectional and Ricci curvature of $(M,\sigma)$. Indeed, if $\Ricc \ge 0$, positive harmonic functions on $M$ are constant, by S.Y. Cheng and S.T. Yau's gradient estimate \cite{yau, chengyau}, while a harmonic function with linear growth forces any tangent cone at infinity of $M$ to split, by work of J. Cheeger, T. Colding and W. Minicozzi \cite{ccm}. \tcb{Furthermore, $M$ itself splits off a line if $\Ricc \ge 0$ is strengthened to $\Sec \ge 0$ (cf. \cite{abfp} for a complete proof), or if $M$ is parabolic (see \cite{litam} and Remark \ref{rem_parabolic} below).} \par
In view of the convergence theory developed in the past 50 years for manifolds with $\Sec \ge 0$ or $\Ricc \ge 0$, some of the tools used to prove the Bernstein theorem in $\R^m$ are available on manifolds satisfying \eqref{ipo_secric}, making these assumptions a natural setting also for the study of $(\ber 1)$. However, much has to be done and $(\ber 1)$ seems very challenging to prove even on manifolds with $\Sec \ge 0$. In fact, we are aware of no results in this direction.\par 
The situation is different for $(\ber 2)$ and $(\ber 3)$, for which, as we shall detail below, the main difficulty is to prove the results by only requiring $\Ricc \ge 0$, arguably the sharp condition for their validity (in this case, however, $(\ber 3)$ has to be suitably weakened, see later).\par
Regarding $(\ber 2)$, after previous work in \cite{rosenbergschulzespruck} by H. Rosenberg, F. Schulze and J. Spruck, a complete answer was obtained by the first, third and fourth authors together with M. Magliaro in \cite{cmmr}, and independently by Q. Ding in \cite{ding} with different methods:

\begin{theorem}\cite{cmmr,ding}
Complete manifolds $M$ with $\Ricc \ge 0$ satisfy $(\ber 2)$, that is, entire positive minimal graphs over $M$ are constant. 
\end{theorem} 

In this paper, we address $(\ber 3)$. In view of the result in \cite{ccm}, it is reasonable to formulate the following 

\begin{conjecture}\label{conj_1}
Let $M$ be a complete manifold with $\Ricc \ge 0$ and possessing a non-constant entire minimal graph with at most linear growth on one side. Then, every tangent cone at infinity of $M$ splits off a line.
\end{conjecture}

The problem seems to be considerably harder compared to the case of harmonic functions. We are aware of only two results in the direction of Conjecture \ref{conj_1}. The first is \cite{dingjostxin}, where Ding, J. Jost and Y. Xin proved that $\R^m$ is the only manifold satisfying the following assumptions:
\begin{equation}\label{assu_djx}
\left\{ \begin{array}{ll}
(\ref{assu_djx}.\alpha) & \Ricc \ge 0, \qquad \disp \lim_{r \ra \infty} \frac{|B_r|}{r^m} > 0 \\[0.3cm]
(\ref{assu_djx}.\beta) & \text{the curvature tensor decays quadratically}
\end{array}\right.
\end{equation}
and admitting an entire, non-constant minimal graph with at most linear growth on one side. Very recently, Ding \cite{ding_new} posted on arXiv a paper where he proved Conjecture \ref{conj_1} on manifolds satisfying the assumptions in $(\ref{assu_djx}.\alpha)$. The bulk of his argument is to show the remarkable property that the isoperimetric inequality, satisfied by $(M,\sigma)$ in view of $(\ref{assu_djx}.\alpha)$, is inherited by the graph of $u$. This allowed Ding to adapt, in a nontrivial way, tools from \cite{bdgm,bg} and from  Cheeger-Colding's theory to reach the goal. We stress that his method heavily depends on the Euclidean volume growth condition in $(\ref{assu_djx}.\alpha)$. \par
In our work, we address Conjecture \ref{conj_1} without requiring the Euclidean volume growth assumption, but rather a mild further curvature condition. To formulate our main result, we first recall the definition of the $\ell$-th Ricci curvature:
\begin{definition}\label{def_Ricc_l}
Let $(M, \sigma)$ be a manifold of dimension $m \ge 2$. For $\ell \in \{1,\ldots, m-1\}$, the $\ell$-th (normalized) Ricci curvature is the function
$$
v \in T_xM \quad \longmapsto \quad \Ric^{(\ell)}(v) \doteq  \inf_{\footnotesize{\begin{array}{c}
\mathcal{W} \le v^\perp \\
\dim \mathcal{W} = \ell
\end{array}}
} \left( \frac{1}{\ell} \sum_{j=1}^\ell \Sec(v \wedge e_j)\right), 
$$
where $\{e_j\}$ is an orthonormal basis of $\mathcal W$. 
\end{definition}
The function $\Ric^{(\ell)}$ interpolates between the sectional and Ricci curvatures, obtained respectively for $\ell = 1$ and, up to the normalization constant $(m-1)$, for $\ell = m-1$. In particular, with our chosen normalization the following implications are immediate:
$$
\begin{array}{l}
\Sec \ge c \ \  \Longrightarrow \ \ \Ric^{(\ell-1)} \ge c \\[0.2cm]
\Longrightarrow \ \ \Ric^{(\ell)} \ge c \ \  \Longrightarrow \ \ \Ric \ge (m-1)c.
\end{array}
$$
Hereafter, given $H \in C([0,\infty))$ and denoting with $r$ the distance from a fixed origin $o \in M$, we use the short-hand notation $\Ric^{(\ell)}(\nabla r) \ge - H(r)$ on $M$ to mean the inequality
$$
\Ric^{(\ell)}\big(\nabla r(x)\big) \ge - H\big(r(x)\big) \qquad \forall\, x \in M \backslash \big(\{o\} \cup \cut(o)\big),
$$
where $\cut(o)$ is the cut-locus of $o$.\\[0.2cm]
\indent A relevant class of manifolds for which rigidity holds without imposing any growth of $u$ is that of parabolic ones. Recall that a manifold $M$ is said to be parabolic if every positive superharmonic function on $M$ is constant. 

\begin{remark}\label{rem_parabolic}
\emph{\tcb{By work of N. Varopoulos \cite{varopoulos} and Li and Yau \cite{liyau_acta}, if $\Ricc \ge 0$ the parabolicity of $M$ is equivalent to 
	\begin{equation}\label{eq_varopo_intro}
	\int^\infty \frac{s \di s}{|B_s|} = \infty,
	\end{equation}
where $B_s$ is a geodesic ball centered at a fixed origin $o$. Indeed, \eqref{eq_varopo_intro} is sufficient for the parabolicity of a complete manifold, independently of any curvature requirement, see \cite{grigoryan}.}
}
\end{remark}

Lastly, we recall that a tangent cone at infinity for a complete (non-compact) manifold $M$ is any metric space obtained as a blow-down of $M$. More precisely, a pointed metric space $(X_\infty,\di_\infty,x_\infty)$, $x_\infty \in X_\infty$, is a tangent cone at infinity for $(M,\sigma)$ if, for some base point $x \in M$ and some sequence $\{\lambda_n\}$ of positive real numbers such that $\lambda_n \to \infty$, one has
$$
	(M,\lambda_n^{-1}\dist_\sigma,x) \to (X_\infty,\di_\infty,x_\infty)
$$
in the pointed Gromov-Hausdorff (pGH) sense. If $(M,\sigma)$ has non-negative Ricci curvature, then tangent cones at infinity exist based at any point $x\in M$, by Gromov's precompactness theorem \cite{gro}. 

We are ready to state
\begin{theorem}\label{teo_main}
Let $(M,\sigma)$ be a complete Riemannian manifold of dimension $m \geq 2$ with 
	\[
	\Ricc \ge 0, 
	\]
and let $u \in C^\infty(M)$ be a non-constant entire solution to  \eqref{P}.
\begin{itemize}
\item[$(i)$] If $M$ is parabolic, then \tcr{it admits a splitting} $M = N \times \R$ with the product metric $\sigma_N + \di s^2$, for some complete manifold $N$ with $\Ricc_N \ge 0$, \tcr{such that} in the variables $(y,s) \in N \times \R$ it holds $u(y,s) = as + b$ for some $a,b \in \RR$.
\item[$(ii)$] If $M$ is non-parabolic and 
	\begin{itemize}
	\item[-] $u$ has at most linear growth on one side; 
	\item[-] there exists an origin $o \in M$ such that, denoting with $r$ the distance from $o$, 
	\begin{equation}\label{eq_curv_m-2}
	\Ricc^{(m-2)}(\nabla r) \geq - \frac{\bar \kappa^2}{1+r^2} \qquad \text{on } \, M,
	\end{equation}
for some constant $\bar \kappa \ge 0$, 
	\end{itemize}
then every tangent cone at infinity of $M$ splits off a line. 
\end{itemize}
\end{theorem}

\begin{remark}
\emph{\tcb{In Case $(i)$, the claimed splitting $M=N\times\R$ for which $u$ is independent of $N$ may not be the unique splitting of the manifold as a product of a line and a complete manifold, as the case of affine graphs on $M=\R^2$ shows.} In Case $(ii)$, since $M$ is non-parabolic then necessarily $m \ge 3$ (see below), so $\Ricc^{(m-2)}$ is well-defined. In the statement of $(ii)$, we also emphasize that tangent cones at infinity may be based at any point of $M$, not necessarily at $o$. 
}
\end{remark}

Case $(i)$ in Theorem \ref{teo_main} is easy to obtain, and might be well-known among specialists. We included it for the sake of completeness. Regarding Case $(ii)$, the curvature condition in \eqref{eq_curv_m-2} is only used to infer that $u$ has bounded gradient on $M$. In other words, as a consequence of our proof we obtain a generalization of Moser's result in \cite{jmoser} to the following

\begin{theorem}\label{teo_moser}
Let $M$ be a complete manifold with $\Ricc \ge 0$. If $u$ is a non-constant solution to  \eqref{P} and $|Du| \in L^\infty(M)$, then every tangent cone at infinity of $M$ splits off a line. 
\end{theorem}

It was already observed in \cite{ccm} that a manifold $M$ with $\Ricc \ge 0$ may not split off any line despite each of its tangent cones at infinity \tcr{does}. A counterexample was constructed in \cite{kasuewashio}, and building on it we get the following result: 

\begin{proposition}\label{teo_counter}
For $m \ge 4$, there exists a complete manifold $M$ with 
	$$
	\Ricc^{(2)} \ge 0, \qquad \Ricc > 0,  \qquad |\Sec| \le \bar \kappa^2
	$$
for some constant $\bar \kappa>0$, which carries a non-constant minimal graph $u : M \to \R$ with $|Du| \in L^\infty(M)$.
\end{proposition} 

Note that $\Ricc^{(m-2)} \ge 0$ and that, \tcr{having positive Ricci curvature,} $M$ does not split off any line. Whence, in assumption $(ii)$ of Theorem \ref{teo_main} the conclusion cannot be strengthened to a splitting of $M$ itself, \tcr{at least if $m \ge 4$.} 

When $\Sec \ge 0$, however, the above phenomenon does not happen. Leaving aside \tcb{dimension} $m=2$, covered by Case $(i)$ in Theorem \ref{teo_main}, we obtain

\begin{corollary}\label{cor_main}
Let $(M,\sigma)$ be a complete Riemannian manifold of dimension $m \geq 3$ satisfying $\Sec \ge 0$. If there exists a non-constant entire solution $u \in C^\infty(M)$ of \eqref{P} with at most linear growth on one side, then \tcr{$M$ admits a splitting} $M = N\times \R$ with the product metric $\sigma_N + \di s^2$, for some complete manifold $N$ with $\Sec_N \ge 0$, \tcr{such that} in the variables $(y,s) \in N\times \R$ it holds $u(y,s) = as + b$ for some $a,b \in \RR$. 
\end{corollary}

The corresponding problem for harmonic functions was also studied by A. Kasue \cite{kasue}. Corollary \ref{cor_main} relates to the results obtained in \cite{liwang_crelle} by P. Li and J. Wang. There, the authors study complete, stable minimal hypersurfaces $\Sigma \to \overline{M}$ properly immersed into a complete manifold $\overline{M}$ whose sectional curvature is non-negative, and prove that either $\Sigma$ has only one end or $\Sigma$ is a totally geodesic cylinder $P \times \R$, for some compact manifold $P$ with non-negative sectional curvature. Our setting falls into their framework, since a minimal graph in $\R \times M$ is stable, properly embedded and $\overline{M} = \R \times M$ has non-negative sectional curvature. However, our conclusion is stronger, since it allows $M$ to have only one end and it also implies a splitting of $M$ itself.

To conclude, we prove the next result for graphs with slower than linear growth on one side, that should be compared to \cite[Theorem 3.6]{dingjostxin} and \cite[Theorem 1.4]{ding_new}.

\begin{theorem}\label{teo_slower}
Let $(M,\sigma)$ be a complete Riemannian manifold of dimension $m \geq 2$ with $\Ricc \ge 0$, and let $u \in C^\infty(M)$ solve \eqref{P} on $M$ and satisfy
	\begin{equation}\label{eq_slower}
	\lim_{r(x) \to \infty} \frac{u_-(x)}{r(x)} = 0,
	\end{equation}
where $u_-(x) = \max\{ -u(x),0\}$. Assume that either
	\begin{itemize}
	\item[-] $M$ is parabolic, or 
	\item[-] $M$ is non-parabolic and there exists an origin $o \in M$ such that, denoting with $r$ the distance from $o$, 
	$$
	\Ricc^{(m-2)}(\nabla r) \geq - \frac{\bar \kappa^2}{1+r^2} \qquad \text{on } \, M,
	$$
for some constant $\bar \kappa \ge 0$. 
	\end{itemize}
Then, $u$ is constant. 
\end{theorem}

\begin{remark}[\textbf{More general curvature bounds}]
\emph{It is natural to wonder whether conditions $\Ricc \ge 0$ or $\Sec \ge 0$ can be weakened still allowing for some rigidity of $u$ and $M$. In this respect we quote \cite[Example 3.1]{bcmmpr}, where the authors constructed a manifold $M$ of dimension $m \ge 3$ with two ends, satisfying
\[
\Sec \ge - \frac{\bar{\kappa}^2}{1+r^2}, \qquad \vol(B_r) \le C r^m
\]
for constants $\bar\kappa,C>0$ and supporting a non-constant, bounded entire minimal graph. On the other hand, the existence of such a solution to \eqref{P} is forbidden if $M$ has asymptotically non-negative sectional curvature and only one end, see \cite{chh_nonexistence}. An interesting class for which one might try to obtain rigidity results is that of manifolds with quadratically decaying (or asymptotically non-negative) Ricci curvature and linear volume growth, see \cite{sormani}.
}
\end{remark}

\subsection*{Strategy of the proof} 

Case $(i)$ in Theorem \ref{teo_main} is a direct consequence of the parabolicity of $(M,\sigma)$, which in our setting can be transplanted to the graph of $u$. In particular, since every surface with $\Ricc \ge 0$ is parabolic, the result holds if $m=2$, so we focus on dimension $m \ge 3$. In \cite{bdgm}, the authors obtain $(\ber 3)$ in $\R^m$ via the following steps: 
	\begin{itemize}
	\item[(a)] a sharp gradient estimate, implying that a solution $u \in C^\infty(\R^m)$ of \eqref{P} with at most linear growth on one side satisfies $|Du| \in L^\infty(\R^m)$;
	\item[(b)] an argument of Moser \cite{jmoser}: since $|Du| \in L^\infty(\R^m)$, for each coordinate field $\partial_j$ the partial derivative $\partial_j u$ is a bounded solution to  a uniformly elliptic PDE. The global Harnack inequality implies that $\partial_j u$ is constant, which implies that $u$ is affine. 
	\end{itemize} 

Step (b) cannot be implemented on manifolds, which in general lack parallel fields. An alternative idea was proposed in \cite{ccm} to study harmonic functions, a blowdown argument which exploits the convergence theory of manifolds with $\Ricc \ge 0$. Our strategy closely follows the one in \cite{ccm}, and can be split into the following steps:
	\begin{itemize}
	\item[(a)] we prove that a solution $u \in C^\infty(M)$ of \eqref{P} with at most linear growth on one side satisfies $|Du| \in L^\infty(M)$;
	\item[(b)] for fixed $x_0 \in M$, we show that the functions $|Du|$ and $|D^2 u|$ satisfy 
	\begin{align}
	\label{Dv}
	& \lim_{R\to\infty} \frac{1}{|B_R(x_0)|} \int_{B_R(x_0)} |Du|^2 \di x = \sup_M |Du|^2, \\
	\label{Hess_v}
    & \lim_{R\to\infty} \frac{R^2}{|B_R(x_0)|} \int_{B_R(x_0)} |D^2 u|^2 \di x = 0,
\end{align}
where $B_R(x_0)$ is the geodesic ball of radius $R$ and center $x_0$ in $(M,\sigma)$.
	\item[(c)] we use the blowdown argument to guarantee the splitting of any tangent cone at infinity with base point $x_0$.  
	\end{itemize}	
To be more precise, Step (a) will be achieved by assuming 
	\begin{equation}\label{eq_ourass}
	\Ricc \ge 0 \qquad \text{and} \qquad \Ricc^{(m-2)}(\nabla r) \ge - \frac{\bar \kappa^2}{1+r^2},
	\end{equation}
while (b) will be shown by requiring 
	\begin{equation}\label{eq_perb}
	\Ricc \ge 0 \qquad \text{and} \qquad |Du| \in L^\infty(M).
	\end{equation}
Though the strategy is the same as that in \cite{ccm}, we emphasize that the techniques in the current literature to prove (a) (respectively, (b)) do not apply under the sole assumptions in \eqref{eq_ourass} (respectively, in \eqref{eq_perb}). We shall justify this claim in the next sections. Our strategy to obtain (a) is to refine a method due to N. Korevaar \cite{korevaar}, see Theorem \ref{teo_gradient_decay} below, while to get (b) in our needed generality we exploit heat equation techniques, inspired by works of P. Li \cite{liann} and L. Saloff-Coste \cite{sal92}. In this respect, we underline Theorem \ref{teo_mvLf} below, yielding to \eqref{Hess_v}, that in the stated generality seems to us new and of an independent interest.

%

\section{Preliminaries}
	
We briefly review some formulas for minimal graphs that will be used later on. In local coordinates $(x^i)$ on $M$, \tcb{the background metric $\sigma$ and the graph metric $g = F^*(\di t^2 + \sigma)$ write as} 
	$$
	\sigma = \sigma_{ij} \di x^i \otimes \di x^j, \quad g = g_{ij} \di x^i \otimes \di x^j, \quad \di u = u_i \di x^i,
	$$
and $g_{ij} = \sigma_{ij} + u_i u_j$. Letting $\sigma^{ij}$ and $g^{ij}$ be the components of the inverse matrices of $(\sigma_{ij})$, $(g_{ij})$, respectively, it holds
	$$
	g^{ij} = \sigma^{ij} - \frac{u^iu^j}{W^2},
	$$
	where $W = \sqrt{1+|Du|^2}$ and $u^i = \sigma^{ij}u_j$. In general, if $\phi \in C^1(M)$ then \tcb{the symbols $D\phi$ and $\nabla\phi$ will denote the gradients of $\phi$ in the metrics $\sigma$ and $g$, respectively,} and in local notation we write
	$$
	\di\phi = \phi_i \di x^i, \qquad D\phi = \phi^i \partial_{x_i} \equiv \sigma^{ij} \phi_j \partial_{x_i}, \qquad \nabla \phi = g^{ij} \phi_j \partial_{x_i}.
	$$
Differentiating the upward pointing unit normal vector ${\bf n} = W^{-1}(\partial_t - u^i e_i)$, the second fundamental form $\II$ in the direction of ${\bf n}$ has components 
	\begin{equation}\label{eq_IIu}
	\II_{ij} = \frac{u_{ij}}{W}, 
	\end{equation}
where $u_{ij}$ are the components of the Hessian $D^2 u$ in the metric $\sigma$. Let $H = g^{ij}h_{ij}$ be the mean curvature, which we assume to vanish. Using the relation 
$$
\Gamma^k_{ij} - \gamma^k_{ij} = \frac{u^k u_{ij}}{W^2}
$$
between the Christoffel coefficients $\Gamma^k_{ij}$ of $g$ and $\gamma^k_{ij}$ of $\sigma$, for every $\phi : M \to \R$ the Laplace-Beltrami operator $\Delta_g$ of $g$ writes as
\[
\Delta_g \phi = g^{ij}\phi_{ij} - \phi_ku^k \frac{H}{W} = g^{ij}\phi_{ij},
\]
where we used the minimality of $\Sigma$. Also, $\Delta_g$ has the following local expression:
	\begin{equation}\label{eq_Lapla}
	\Delta_g \phi = \tcb{\frac{1}{\sqrt{|g|}} \partial_{x_j} \left( \sqrt{|g|} g^{ij} \phi_i \right)} = \frac{1}{W} \diver (W g^{ij} \phi_i \partial_{x_j})
	\end{equation}
	where $\diverge$ is, as before, the divergence operator in $(M,\sigma)$ and \tcb{$|g|$ is the determinant of $(g_{ij})$}. 
Next, for every Killing field $\bar X$ defined in $\R \times M$, the angle function $\Theta_{\bar X} \doteq \langle {\bf n}, \bar X \rangle$ solves the Jacobi equation
\begin{equation} \label{Jeq}
	\Delta_g \Theta_{\bar X} + \Big( \|\II\|^2 + \overline{\Ricc}({\bf n},{\bf n})\Big) \Theta_{\bar X} \tcr{=0},
\end{equation}
with $\overline{\Ricc}$ the Ricci curvature of $\R \times M$. This is the case, for instance, of the angle function $\Theta_{\partial_t} = \langle \mathbf{n},\partial_t \rangle = W^{-1}$ associated to the Killing field $\partial_t$. As a consequence, $W$ satisfies
\begin{equation}\label{eq_W}
	\LL_W W = \left(\|\II\|^2 + \overline{\Ricc}({\bf n}, {\bf n})\right) W,
\end{equation}
where we defined 
	\[
	\LL_W \phi \doteq W^{2}\diverge_g \big( W^{-2} \nabla \phi) = \Delta_g \phi - 2 \langle \nabla \log W, \nabla \phi \rangle.
	\]
Observe that, in terms of the metric $\sigma$, 
	\begin{equation}\label{def_LW}
	\LL_W \phi = W \diver \big( W^{-1} g^{ij} \phi_i \partial_{x_j}\big).
	\end{equation}
If $X$ is a Killing field in $(M,\sigma)$ then we can extend it by parallel transport on $\R\times M$ to a Killing field $\bar X$ satisfying $\langle \partial_t,\bar X \rangle = 0$, with corresponding angle function $\Theta_{\bar X} = \langle \mathbf{n}, \bar X \rangle = -W^{-1} \sigma(Du, X)$. Since \eqref{Jeq} holds both for $\Theta_{\partial_t}$ and for $\Theta_{\bar X}$, it can be checked that the quotient 
	\[
	v \doteq -\Theta_{\bar X} / \Theta_{\partial_y} = \sigma (Du, X)
	\]
is a solution to 
\begin{equation} \label{killing-Du}
	\LL_W v = 0.
\end{equation}

We next discuss the implications of $\ell$-th Ricci curvature lower bounds. Hereafter, we set $\R^+ = (0,\infty)$ and $\R^+_0 = [0,\infty)$. 

\begin{proposition}\label{prop_distances}
Let $(M, \sigma)$ be a complete manifold of dimension $m \ge 2$ satisfying 
	\[
	\Ricc^{(\ell)}(\nabla r) \ge - H(r) \qquad \text{for } \, \ell = \max\{1,m-2\}, 
	\]
where $r$ is the distance from a fixed origin $o \in M$, and $0 \le H \in C(\RR^+_0)$. Let $h \in C^2(\R^+_0)$ solve 
	\begin{equation}\label{eq_compa}
	\left\{ \begin{array}{lcl}
	h'' - H h \ge 0 \qquad \text{on } \, \R^+, \\[0.2cm]
	\disp \liminf_{t \to 0} \left( \frac{h'}{h} - \frac{1}{t}\right) \ge 0.
	\end{array}\right.
	\end{equation}
Let $u : M \to \R$ solve \eqref{P}. Then, denoting with $\Delta_g$ the Laplacian in the graph metric $g$, 
	\[
	\Delta_g r \le m\frac{h'(r)}{h(r)} 
	\] 
pointwise on $M \backslash (\{o\} \cup \cut(o))$ and in the barrier sense on $M \backslash \{o\}$. 
\end{proposition}

\begin{proof}
Outside of $\{o\} \cup \cut(o)$, denote by $\{\lambda_j(D^2r)\}$ the eigenvalues of $D^2r$ in increasing order. The comparison theorem in \cite[Prop. 7.4]{maripessoa} guarantees that 
	\begin{equation}\label{eq_Pkp}
	\sum_{j = 2}^m \lambda_j(D^2r) \le (m-1)\frac{h'(r)}{h(r)}
	\end{equation}
pointwise on $M \backslash (\{o\} \cup \cut(o))$ and in the barrier sense on $M \backslash \{o\}$. Note that the initial assumptions on $h$ therein are $h(0) = 0$, $h'(0) \ge 1$, but the same proof works for the more general \eqref{eq_compa}. In this respect, note that $h'>0$ on $\R^+$ follows from $H \ge 0$. \par
To estimate $\Delta_g r = g^{ij}r_{ij}$, pick a point $x$ where $r$ is smooth. If $Du(x) = 0$, then $g^{ij} = \tcr{\sigma}^{ij}$. In our assumptions, the Ricci curvature satisfies 
	\[
	\Ricc(\nabla r,\nabla r) \ge - (m-1)H(r),
	\]
so by the Laplacian comparison theorem and since $h'>0$, 
	\[
	\Delta_g r = \tr(D^2 r) \le (m-1)\frac{h'(r)}{h(r)} \le m \frac{h'(r)}{h(r)}. 
	\] 
Assume that $Du(x) \neq 0$, write $\nu = Du/|Du|$ in a neighbourhood of $x$ and complete it to a local \tcr{$\sigma$-}orthonormal basis $\{\nu, e_\alpha\}$ with $2 \le \alpha \le m$. Note that $g^{ij}$ is diagonalized with eigenvalues $W^{-2} \le 1$ in direction $\nu$ and $1$ in directions $\{e_\alpha\}$. Expressing $\Delta_g r$ in the basis $\{e_\alpha, \nu\}$ we get
\begin{equation}\label{hessvarrho}
\begin{array}{lcl}
\Delta_g r & = & \disp \frac{1}{W^2} D^2r(\nu,\nu) + \sum_{\alpha=2}^m D^2r(e_\alpha,e_\alpha) \\[0.4cm]
& = & \disp \frac{1}{W^2} \left[ \tr(D^2r) - \sum_{\alpha=2}^m D^2r(e_\alpha,e_\alpha)\right] + \sum_{\alpha=2}^m D^2r(e_\alpha,e_\alpha) \\[0.4cm]
& = & \disp \frac{1}{W^2} \tr(D^2 r) + \left[ \frac{W^2-1}{W^2}\right]\sum_{\alpha=2}^m D^2r(e_\alpha,e_\alpha).  
\end{array}
\end{equation}
By min-max and since the eigenvalues are ordered, 
$$
\tr(D^2 r) \le \frac{m}{m-1} \sum_{\alpha=2}^m \lambda_\alpha(D^2r), \qquad \sum_{\alpha=2}^m D^2 r(e_\alpha,e_\alpha) \le \sum_{\alpha=2}^m \lambda_\alpha(D^2r).
$$
Therefore, 
\begin{equation}\label{hessvarrho}
\begin{array}{lcl}
\Delta_g r & \le & \disp \left[ \frac{m}{(m-1)W^2} + \frac{W^2-1}{W^2}\right] \sum_{\alpha=2}^m \lambda_\alpha(D^2r) \\[0.4cm]
& \le & \disp \left[ \frac{1}{(m-1)W^2} + 1 \right](m-1)\frac{h'(r)}{h(r)} \le m\frac{h'(r)}{h(r)},
\end{array} 
\end{equation}
as claimed. \tcb{The validity of \eqref{eq_Pkp} in the barrier sense on the entire $M \backslash \{o\}$ easily follows by Calabi's trick,} see \cite[Prop. 7.4]{maripessoa}.	
\end{proof}

\begin{remark}\label{rem_distances}
\emph{In particular, letting $\bar \kappa \in \R^+_0$, for $\ell = \max\{1, m-2\}$ it holds 
	\[
	\begin{array}{lcl}
	\Ricc^{(\ell)}(\nabla r) \ge - \bar \kappa^2 & \qquad \Longrightarrow \qquad & \qquad \Delta_g r \le m\bar \kappa \coth(\bar \kappa r) \\[0.2cm]
	\Ricc^{(\ell)}(\nabla r) \ge - \frac{\bar \kappa^2}{1+ r^2} & \qquad \Longrightarrow \qquad & \qquad \Delta_g r \le \frac{m(1 + \sqrt{1+4 \bar \kappa^2})}{2 r} 
	\end{array}
	\]
pointwise outside of $\{o\} \cup \cut(o)$ and in the barrier sense on $M \backslash\{o\}$. Indeed, it is enough to consider for $h$, respectively, the functions
	\[
	h(t) = \frac{\sinh(\bar \kappa t)}{\bar \kappa} \qquad \text{and} \qquad h(t) = t^{\bar \kappa'}, \ \text{ where } \, \bar \kappa' = \frac{1 + \sqrt{1+4 \bar \kappa^2}}{2}.
	\]
}
\end{remark}

\section{Proof of Theorem \ref{teo_main}, $(i)$}

\tcb{Since $\Ricc \ge 0$ and $M$ is parabolic, by Remark \ref{rem_parabolic} the manifold $(M,\sigma)$ satisfies
	\begin{equation}\label{eq_varopo}
	\int^\infty \frac{s \di s}{|B_s|} = \infty.
	\end{equation}
We apply an argument outlined in \cite[p.~48]{cm_book} for minimal graphs in $\R^3$. First, a calibration method in \cite{liwang_mini}} (cf.~also \cite{cm_book,trudi} for the case $M=\R^m$) shows that the volume of the graph $\Sigma = (M,g)$ inside an extrinsic ball $\BB_r \subset \R \times M$ centered at a point $(u(o),o)$ satisfies
	\[
	|\Sigma \cap \BB_r| \le |B_r| + \frac{1}{2} |B_{3r}\backslash B_r| \le 2|B_{3r}|.	
	\]
Hence, the volume of a geodesic ball $B^g_s$ in $\Sigma$ centered at $o$ is bounded as follows:
	\begin{equation}\label{finiteden}
	|B_r^g| \le |\Sigma \cap \BB_r| \le  2|B_{3r}|,
	\end{equation}
which implies
	\[
	\int^\infty \frac{s \di s}{|B^g_s|} = \infty.
	\]
Therefore, by Remark \ref{rem_parabolic}, the graph $\Sigma = (M,g)$ is parabolic. Because of the Jacobi equation
$$
\Delta_g \frac{1}{W} = - \left( \|\II\|^2 + \overline{\Ricc}({\bf n},{\bf n}) \right) \frac{1}{W},
$$
the bounded function $1/W$ is superharmonic on $\Sigma$, hence constant by parabolicity. \tcb{Since $\Ricc \ge 0$ implies $\overline{\Ricc} \ge 0$, again from the Jacobi equation we deduce $\II\equiv 0$ and $\Sigma$ is totally geodesic in $M \times \R$. Equivalently, by \eqref{eq_IIu}, $D^2 u \equiv 0$ in $M$. As a consequence, since $u$ is non-constant, $Du$ is a non-zero parallel vector field. The flow of $Du$ therefore splits $M$ isometrically as a product $N \times \R$, and $u$ is an affine function of the $\R$-coordinate alone.}


\section{A local gradient estimate}\label{sec_localgrad}

%

Let $u : B_R \subset M \to \R$ solve \eqref{P} on a geodesic ball $B_R= B_R(x)$. The original argument in \cite{bdgm} to prove the gradient estimate in Euclidean setting \tcr{($M = \R^m$)} 
	\begin{equation}\label{eq_grad_bdgm}
	|Du(x)| \le c_1\exp\left\{ c_2 \frac{u(x)- \inf_{B_R} u}{R} \right\}, 
	\end{equation}
for some constants $c_j = c_j(m)$, makes use of the isoperimetric inequality, which does not hold for minimal graphs over manifolds $(M,\sigma)$ with $\Ricc \ge 0$ unless $M$ has maximal volume growth compatible with the Bishop-Gromov inequality, in the sense that $(\ref{assu_djx}.\alpha)$ is satisfied. Indeed, the isoperimetric inequality forces geodesic balls $B_r^g$ in the graph $\Sigma = (M,g)$ to satisfy $|B_r^g| \ge Cr^m$, which coupled with \eqref{finiteden} imply that $M$ has Euclidean volume growth. 

We mention that, in the seminal paper \cite{bg}, \tcr{Bombieri and Giusti proved a different estimate for entire solutions: if $u : \RR^m \to \RR$ solves \eqref{P}, then for any $x\in\RR^m$ and $R>0$}
	\begin{equation}\label{eq_grad_bg}
	|Du(x)| \le c_1 \left\{ 1 + \frac{\sup_{B_R}|u|}{R} \right\}^m, 
	\end{equation}
\tcr{where $c_1 = c_1(m)$. Whence, for entire solutions,} an exponentially growing bound in terms of $|u|$ is not sharp.
%
%
%
%
%
%

If $M$ has $\Ricc \ge 0$ and Euclidean volume growth, the validity of an isoperimetric inequality on entire minimal graphs was recently shown in \cite{ding_new}, see also \cite{brendle} for the case $\Sec \ge 0$. As a consequence, in \cite[Theorems 1.3 and 6.2]{ding_new} the author was able to extend \eqref{eq_grad_bdgm} and \eqref{eq_grad_bg} to such manifolds.\par
An alternative method to prove \eqref{eq_grad_bdgm} in Euclidean setting was given by N. Trudinger \cite{trudi}. His strategy hinges on a mean value inequality on $\Sigma$ which, remarkably, is obtained without needing the isoperimetric inequality and is therefore suited to apply to manifolds \tcb{whose volume growth is not Euclidean.} However, to adapt the proof to minimal graphs over $M$, it seems that an upper bound on the sectional curvature of $M$ is necessary, see also the related \cite{dingjostxin}. \tcb{Later, N. Korevaar in \cite{korevaar} gave new insight into the problem, finding a striking argument to get gradient estimates that only requires lower bounds on the curvatures of $M$.} Exploiting Korevaar's method, in \cite{rosenbergschulzespruck} the authors obtained the slightly different estimate 
	\begin{equation}\label{eq_grad_rss}
	|Du(x)| \le c_1\exp\left\{ c_2[1 + \bar \kappa R \coth( \bar\kappa R)] \frac{(u(x)- \inf_{B_R} u)^2}{R^2} \right\}, 
	\end{equation}
provided that $\Ricc \ge 0$ and $\Sec \ge - \bar \kappa^2$. Note that, unless $\bar \kappa =0$, the estimate explodes as $R \to \infty$ if $u : M \to \R$ is of linear growth. Extensions to more general ambient spaces were later given in \cite{chh_nonexistence,dl_bounded,dl_unbounded}, but they only consider graphs which are bounded on one side or have logarithmic growth.

Inspecting the proofs in \cite{rosenbergschulzespruck,dl_bounded,dl_unbounded,dingjostxin}, to reach the inequality $|Du(x)|\le C$ for solutions of linear growth, \tcr{with $C$ uniform with respect to $x$,} the bounds on $\Sec$ are instrumental to guarantee that the distance $r_{x}$ from $x$ satisfies $\Delta_g r_{x} \le C_1/r_x$ for some absolute constant $C_1$. In view of the arbitrariness of the point $x$, assumption $\Sec \ge 0$ in \cite{rosenbergschulzespruck} seems therefore difficult to replace \tcb{by} a weaker control on $\Sec$ from below. For instance, if one considers the inequality $\Sec \ge - \bar \kappa^2/(1+r_o^2)$ for some constant $\bar \kappa >0$ and some origin $o$, comparison theory and standard estimates for ODE would yield to a constant $C_1$, hence $C$, that depends on the distance of $x$ from $o$ and explodes as $r_o(x) \to \infty$, making the estimate on $\Delta_g r_x$ insufficient to imply the desired \tcr{uniform} gradient bound.

	From a different perspective, we mention that a \emph{global} gradient estimate for \emph{positive} entire solutions was obtained  in \cite{cmmr} under the sole curvature assumption 
	\[
	\Ricc \ge - (m-1)\kappa^2, \qquad \kappa \in \R^+_0, 
	\]
namely, a positive solution to  \eqref{P} on the entire $M$ shall satisfy
	\[
	\sqrt{1+|Du(x)|^2} \le e^{\kappa u(x)\sqrt{m-1}} \qquad \forall \, x \in M.
	\]	
Note that $(\ber 2)$ directly follows if $\kappa = 0$. However, modifying the argument in \cite{cmmr} to allow for linearly growing solutions seems challenging.

Our first main result, Theorem \ref{teo_gradient_decay}, provides an improvement of Korevaar's method that apply to the more general assumption \eqref{eq_ourass}.


\begin{theorem}\label{teo_gradient_decay}
	Let $(M,\sigma)$ be a complete Riemannian manifold with dimension $m \geq 2$. Let $B_R = B_R(o) \subseteq M$ be a geodesic open ball of radius $R > 0$ centered at $o \in M$ and let $u \in C^3(\overline{B_R})$ be a non-constant solution to \eqref{P}. Assume that
	$$
	\Ricc \geq - (m-1)\kappa^2, \qquad \Ricc^{(\ell)}(\nabla r) \geq - \frac{\bar \kappa^2}{1+r^2} \qquad \text{on } \, B_R, \ \ \ell = \max\{1,m-2\},
	$$
	for some $\kappa, \bar \kappa \in \RR^+_0$, where $r$ denotes the distance from $o$. Let $0 < R_1 < R$. Then, 
	$$
	\begin{array}{l}
	\sqrt{1+|Du(x)|^2} \\[0.2cm]
	\disp \qquad \leq \max\left\{\sqrt{1+a_0^2(\gamma^\ast)^2},\sqrt{\frac{a_3}{a_3-a_2}}\right\} \left( \frac{e^{LR(\sqrt{\eps^2 + 1} - \eps)} - 1}{e^{LR(\sqrt{\eps^2 + 1} - \sqrt{\eps^2 + r(x)^2/R^2} - q\gamma(x))} - 1} \right)
	\end{array}
	$$
	for every $x \in B_{R_1}(o)$, where
	$$
	\gamma(x) = \frac{u(x) - \inf_{B_R} u}{R} \, , \qquad \gamma^\ast = \sup_{x\in B_{R_1}} \gamma(x) = \frac{\sup_{B_{R_1}} u - \inf_{B_R} u}{R} \, ,
	$$
	$\eps>0$ and $\tau \in (0,1)$ are fixed arbitrarily, $q,a_0 \in \RR^+$ satisfy
	\[ 
	\frac{\sqrt{1+\eps^2}-\sqrt{(R_1/R)^2 + \eps^2}}{\gamma^\ast} > q > \frac{1}{\sqrt{\tau} a_0 \gamma^\ast} > 0
	\]
	and $L \in \RR^+$ satisfies
	$$
	(1-\tau)\left(q^2 - \frac{1}{\tau a_0^2(\gamma^\ast)^2}\right)L^2 - \frac{(m+1)\bar \kappa_0 L}{\eps R} > (m-1)\kappa^2
	$$
	with $\bar \kappa_0 = \max\{1, \bar \kappa\}$. Finally, $a_2,a_3$ are defined by
	$$
	a_2 = \frac{(m+1)\bar \kappa_0 L}{\eps R} \, \qquad a_3 = (1-\tau)\left(q^2 - \frac{1}{\tau a_0^2(\gamma^\ast)^2}\right)L^2 - (m-1)\kappa^2 \, .
	$$
\end{theorem}

\begin{remark}
\emph{The assumption that $u$ is non-constant ensures that $\gamma_\ast > 0$ by the maximum principle.
}
\end{remark}

Before proving the Theorem, we give some applications, starting from the case where $\kappa=0$.

\begin{corollary}
	Let $(M^m,\sigma)$ be a complete Riemannian manifold with $\Ricc\geq0$ and 
	\[
	\Ricc^{(\ell)}(\nabla r) \geq -\frac{\bar \kappa^2}{1+r^2}, \qquad \ell = \max\{1,m-2\},
	\]
	for some $\bar \kappa \in \RR^+_0$, where $r$ denotes the distance from $o$. Let $u \in C^3(\overline{B_R})$ solve \eqref{P}. Then, for every $\delta\in(0,1)$ and for every $R_1 \in (0,\delta R)$,
	\begin{equation} \label{BR1_est}
		\sup_{B_{R_1}}\sqrt{1+|Du|^2} \leq C_1 \exp\left( C_2 m \bar \kappa_0 \frac{[\sup_{B_{R_1}} u - \inf_{B_R} u]^2}{R^2} \right)
	\end{equation}
	with $\bar \kappa_0=\max\{1, \bar \kappa\}$ and $C_1,C_2>0$ only depending on $\delta$.
\end{corollary}

\begin{proof}
The desired inequality is trivial if $u$ is constant, so assume that $u$ is non-constant. It suffices to prove the claim for $\delta\in[1/2,1)$. Let
	$$
		\gamma^\ast = \frac{\sup_{B_{R_1}}u - \inf_{B_R}u}{R} \, .
	$$
	Choose
	$$
		\tau = \frac{1}{2} \, , \qquad \eps = \delta \, , \qquad q = \frac{1-\delta}{2\sqrt{2}\gamma^\ast} \, , \qquad a_0 = \frac{2}{q\gamma^\ast} \, , \qquad L = \frac{8(m+1)\bar\kappa_0}{\delta R q^2} \, .
	$$
	With this choice, we have
	\begin{equation}\label{eq_fir}
	\begin{array}{lcl}
	\disp \frac{\sqrt{1+\eps^2}-\sqrt{(R_1/R)^2 + \eps^2}}{\gamma^\ast} & \geq & \disp \frac{\sqrt{1+\eps^2}-\sqrt{\delta^2 + \eps^2}}{\gamma^\ast} \\[0.4cm]
	& = & \disp \frac{\sqrt{1+\delta^2} - \sqrt{2}\delta}{\gamma^\ast} \geq 2q,
	\end{array}
	\end{equation}
	where, from $\delta < 1$, we used
	$$
		\sqrt{1+\delta^2} - \sqrt{2}\delta = \frac{1-\delta^2}{\sqrt{1+\delta^2} + \sqrt{2}\delta} \geq \frac{1-\delta^2}{\sqrt{2} + \sqrt{2}\delta} = \frac{1-\delta}{\sqrt{2}}.
	$$
	We also have
	$$
		q^2 - \frac{1}{a_0^2 (\gamma^\ast)^2 \tau} = q^2 - \frac{q^2}{2} = \frac{q^2}{2}
	$$
	and then
	$$
		a_3 = (1-\tau)\left( q^2 - \frac{1}{a_0^2(\gamma^\ast)^2\tau} \right) L^2 = \frac{L^2 q^2}{4} = 2 \frac{(m+1)\bar\kappa_0 L}{\eps R} = 2a_2 \, .
	$$
	Hence, all assumptions of Theorem \tcr{\ref{teo_gradient_decay}} are satisfied and for every $x\in B_{R_1}$ we have
\[
\sqrt{1+|Du(x)|^2} \leq \max\left\{ \sqrt{1 + a_0^2(\gamma^\ast)^2} \, , \sqrt{2} \right\} \cdot \frac{e^{LR(\sqrt{1+\eps^2} - \eps)} - 1}{e^{LR(\sqrt{1+\eps^2} - \sqrt{(r(x)/R)^2 + \eps^2} - q\gamma(x))} - 1} \, .
\]
	Note that, for every $x \in B_{R_1}$ and taking into account \eqref{eq_fir}, 
	\begin{align*}
		\sqrt{1+\eps^2} - \sqrt{(r(x)/R)^2 + \eps^2} - q\gamma(x) & \geq \sqrt{1+\eps^2} - \sqrt{\delta^2 + \eps^2} - q\gamma^\ast \\
		& \geq 2q\gamma^\ast - q\gamma^\ast = q\gamma^\ast,
	\end{align*}
	and also
	$$
		\sqrt{1+\eps^2} - \eps = \frac{1}{\sqrt{1+\eps^2} + \eps} \leq \frac{1}{\sqrt{1+\eps^2}} = \frac{1}{\sqrt{1+\delta^2}} \leq \frac{1}{\sqrt{2}\delta} = \frac{2}{\delta(1-\delta)} q\gamma^\ast.
	$$
	Therefore, we can estimate
	$$
		\frac{e^{LR(\sqrt{1+\eps^2} - \eps)} - 1}{e^{LR(\sqrt{1+\eps^2} - \sqrt{(r(x)/R)^2 + \eps^2} - q\gamma(x))} - 1} \leq \frac{e^{LR\frac{2}{\delta(1-\delta)}q\gamma^\ast} - 1}{e^{LRq\gamma^\ast} - 1} \leq C(\delta)e^{LR\left(\frac{2}{\delta(1-\delta)}-1\right)q\gamma^\ast} \, .
	$$
	Here, we have exploited the fact that for every $\alpha\in\RR$ one has the validity of an inequality of the form
	$$
		\frac{y^\alpha - 1}{y-1} \leq C(\alpha) y^{\alpha-1} \qquad \forall \, y > 1
	$$
	for a suitable constant $C(\alpha)>0$. Recalling that
	\begin{align*}
		a_0^2(\gamma^\ast)^2 & = \frac{32(\gamma^\ast)^2}{(1-\delta)^2} \, , \\
		LRq\gamma^\ast & = \frac{8(m+1)\bar\kappa_0\gamma^\ast}{\delta q} = \frac{16\sqrt{2}(m+1)\bar\kappa_0}{\delta(1-\delta)} (\gamma^\ast)^2
	\end{align*}
	we obtain
	\begin{align*}
		\sqrt{1+|Du(x)|^2} & \leq \max\left\{ \sqrt{1 + \frac{32(\gamma^\ast)^2}{(1-\delta)^2}} \, , \sqrt{2} \right\} \cdot \\
		& \qquad \qquad \cdot C(\delta) \exp\left( \frac{16\sqrt{2}(m+1)\bar\kappa_0}{\delta(1-\delta)}\left(\frac{2}{\delta(1-\delta)}-1\right) (\gamma^\ast)^2 \right)
	\end{align*}
	and the conclusion follows.
\end{proof}

Assuming that $u$ has at most linear growth, we get

\begin{corollary}\label{cor_boundedgrad}
	Let $(M^m,\sigma)$ be a complete Riemannian manifold with $\Ricc\geq0$ and 
	\[
	\Ricc^{(\ell)}(\nabla r) \geq-\frac{\bar \kappa^2}{1+r^2}, \qquad \ell = \max\{1,m-2\},
	\]
	for some $\bar \kappa \in \RR^+_0$. If $u \in C^\infty(M)$ solves \eqref{P} and has at most linear growth on one side, then $|Du| \in L^\infty(M)$.
\end{corollary}

\begin{proof}
	Without loss of generality we can assume that the negative part of $u$ has at most linear growth, so that there exists $a>0$ such that $u(x) \geq -a(1+r(x))$ for every $x\in M$. Let $R_1 > 0$ be fixed. Choosing $\delta = 1/2$ and letting $R\to \infty$ in estimate \eqref{BR1_est} we get
	$$
		\sup_{B_{R_1}} \sqrt{1+|Du|^2} \leq C_1 \exp\left( C_2 m\bar\kappa_0 a^2 \right)
	$$
	where $C_1,C_2>0$ do not depend on $R_1$. Since $R_1>0$ was arbitrary, the conclusion follows. 
\end{proof}

To prove Theorem \ref{teo_gradient_decay}, we need the following

\begin{lemma}\label{lem_deltapsi_decay}
	Let $(M^m,\sigma)$ be a complete Riemannian manifold with
	$$
		\Ricc^{(\ell)}(\nabla r) \geq - \frac{\bar \kappa^2}{1+r^2}, \qquad \ell = \max\{1, m-2\},
	$$
	for some $\bar \kappa \in \RR^+_0$, where $r$ is the distance from a fixed origin $o\in M$, and let $u \in C^\infty(B_R)$ solve \eqref{P} on a geodesic ball $B_R$ centered at $o$.
	
	For any given $a>0$, the function $\psi = \sqrt{a^2 + r^2}$ satisfies
	$$
		|D\psi| < 1, \qquad \Delta_g \psi \leq (m+1) \frac{\max\{1,\bar\kappa\}}{a} 
	$$
	in the barrier sense on $B_R$ and pointwise on $B_R\setminus\cut(o)$.
\end{lemma}

\begin{remark}
\emph{By its very definition, a solution in the barrier sense is also a solution in the viscosity sense, see \cite{mmu} for comments.
}
\end{remark}

\begin{proof}
	Outside of $\{o\} \cup \cut(o)$, a direct computation yields $|D\psi| = \frac{r}{\psi}|Dr| < 1$ and
	$$
		\Delta_g \psi = \frac{r \Delta_g r}{\sqrt{a^2+r^2}}  + \frac{a^2\|\nabla r\|^2}{(a^2 + r^2)^{3/2}}.
	$$
	From $\|\nabla r\|^2 = g^{ij}r_ir_j \le |Dr|^2 = 1$ and Remark \ref{rem_distances},
	\begin{align*}
		\Delta_g \psi & \le \frac{1}{\sqrt{a^2 + r^2}} \left( \frac{m(1+\sqrt{1+4\bar\kappa^2})}{2} + \frac{a^2}{a^2 + r^2} \right)
	\end{align*}
	and the conclusion follows by observing that $\frac{1}{\sqrt{a^2 + r^2}} \leq \frac{1}{a}$ and that
	$$
		\frac{m(1+\sqrt{1+4\bar\kappa^2})}{2} + \frac{a^2}{a^2 + r^2} \leq m(1+\bar\kappa) + 1 \leq (m+1) \max\{1,\bar\kappa\} \, .
	$$
	The validity of the inequality in the barrier sense can be proved by Calabi's trick, see for instance \cite[Prop. 7.4]{maripessoa}.
\end{proof}

\begin{proof}[Proof of Theorem \ref{teo_gradient_decay}]
	Without loss of generality, we can assume $\inf_{B_R} u = 0$. Then
	$$
	\gamma^\ast = \sup_{x\in B_{R_1}} \gamma(x) = \frac{\sup_{B_{R_1}}u}{R} \, .
	$$
	As in the statement of the theorem, fix $\tau \in (0,1)$ and $\eps>0$, choose $q>0$ and $a_0>0$ such that
	\begin{equation} \label{par_1b}
		\frac{\sqrt{\eps^2+1}-\sqrt{(R_1/R)^2 + \eps^2}}{\gamma^\ast} > q > \frac{1}{\sqrt{\tau} a_0 \gamma^\ast}
	\end{equation}
	and then $L>0$ which satisfies
	\begin{equation} \label{par_2}
		(1-\tau)\left(q^2 - \frac{1}{\tau a_0^2(\gamma^\ast)^2}\right)L^2 - \frac{(m+1) \bar \kappa_0 L}{\eps R} > (m-1)\kappa^2 \, ,
	\end{equation}
	where $\bar\kappa_0 = \max\{1,\bar\kappa\}$. Set
	$$
	C = qL \, , \qquad \delta = e^{-LR\sqrt{\eps^2+1}}, 
	$$
	define the function
	$$
	\psi = \sqrt{\eps^2 R^2 + r^2}
	$$
	where $r(x) = \dist_\sigma(o,x)$, and let
	$$
	\eta = e^{-Cu-L\psi} - \delta \, , \qquad z = W \eta \, .
	$$
	By writing
	$$
	\eta = \delta \left( e^{LR(\sqrt{\eps^2+1} - \sqrt{\eps^2 + (r/R)^2} - qu/R)} - 1 \right)
	$$
	we see that for every $x\in B_{R_1}$
	\begin{align*}
		\eta(x) & = \delta \left( e^{LR(\sqrt{\eps^2+1} - \sqrt{\eps^2 + r(x)^2/R^2} - q\gamma(x))} - 1 \right) \\
		& \geq \delta \left( e^{LR(\sqrt{1+\eps^2} - \sqrt{(R_1/R)^2 + \eps^2} - q\gamma^\ast)} - 1 \right) > 0
	\end{align*}
	as a consequence of \eqref{par_1b}. Noting that, on $\partial B_R$, 
	$$
	\eta = \delta \left( e^{-qLu} - 1 \right) \le 0,
	$$
	the set 
	$$
	\Omega = \{ x \in \overline{B_R} : z(x) > 0 \} \equiv \{ x \in \overline{B_R} : \eta(x) > 0 \}
	$$
	is non-empty and satisfies $B_{R_1} \subseteq \Omega \subseteq B_R$. Therefore, there exists $x_0 \in \Omega$ such that
	$$
	0 < z(x_0) = \max_\Omega z.
	$$
	The function $z$ satisfies
	$$
	\Delta_g z - 2\langle\nabla z,\nabla \log W\rangle \geq \left( -(m-1)\kappa^2\|\nabla u\|^2 + \frac{\Delta_g\eta}{\eta}\right) z \qquad \text{on } \, \Omega.
	$$
	The above inequality has to be interpreted in the viscosity sense, in case $x_0$ is not a point where $r$ (hence, $\psi$) is smooth. By the maximum principle, necessarily
	\begin{equation}\label{main-inequality}
		-(m-1)\kappa^2 \|\nabla u\|^2 + \frac{\Delta_g \eta}{\eta} \leq 0 \qquad \text{at } \, x_0 \, 
	\end{equation}
	in the viscosity sense. We compute
	$$
	\Delta_g\eta = (\eta+\delta)\left( -C\Delta_g u - L\Delta_g\psi + \| C\nabla u + L\nabla\psi \|^2 \right).
	$$
	We recall that $W^{-2}(\sigma^{ij})_{i,j} \leq (g^{ij})_{i,j} \leq (\sigma^{ij})_{i,j}$ in the sense of quadratic forms, hence
	\begin{align*}
		\| C\nabla u + L\nabla\psi \|^2 & = g^{ij} (Cu_i+L\psi_i)(Cu_j+L\psi_j) \\
		& \geq \frac{1}{W^2} \sigma^{ij}(Cu_i+L\psi_i)(Cu_j+L\psi_j) \\
		& \geq \frac{1}{W^2}|C Du + L D\psi|^2 \, .
	\end{align*}
	It follows that
	$$
	\frac{\Delta_g\eta}{\eta+\delta} \geq  -C\Delta_g u - L\Delta_g\psi + \frac{1}{W^2}|C Du + L D\psi|^2 .
	$$
	Using $\Delta_g u = 0$ and Young's inequality we obtain
	$$
	\frac{\Delta_g\eta}{\eta+\delta} \geq -L\Delta_g\psi+(1-\tau)C^2\frac{|D u|^2
	}{W^2}-L^2\frac{1-\tau}{\tau}\frac{|D \psi|^2}{W^2}.
	$$
	Taking into account Lemma \ref{lem_deltapsi_decay} we infer
	\[
	|D\psi| < 1, \qquad \Delta_g \psi \le \frac{(m+1)\bar\kappa_0}{\eps R} \, .
	\]
	Substituting these estimates in the above inequality, we deduce
	$$
	\frac{\Delta_g\eta}{\eta+\delta} \geq (1-\tau)C^2\frac{|Du|^2}
	{W^2}-L\left(\frac{(m+1)\bar\kappa_0}{\epsilon R}+\frac{1-\tau}{\tau}\frac{L}{W^2}\right)
	$$
	If $|Du(x_0)| \geq a_0 \gamma^\ast$ then $\frac{|Du(x_0)|^2}{W^2a_0^2(\gamma^\ast)^2} \geq \frac{1}{W^2}$. Thus, we can further estimate
	$$
	\frac{\Delta_g\eta}{\eta+\delta} \geq (1-\tau)\left(q^2 - \frac{1}{\tau a_0^2 (\gamma^\ast)^2} \right) L^2 \frac{|D u|^2}
	{W^2} - \frac{(m+1)\bar\kappa_0L}{\eps R} \qquad \text{at } \, x_0
	$$
	that is,
	\begin{equation} \label{zeroth-estimate}
		\frac{\Delta_g\eta}{\eta+\delta} \geq a_1 \|\nabla u\|^2 - a_2
	\end{equation}
	with
	$$
	a_1 = (1-\tau)\left(q^2 - \frac{1}{\tau a_0^2(\gamma^\ast)^2} \right) L^2> 0 \, , \qquad a_2 = \frac{(m+1)\bar\kappa_0L}{\eps R} >0.
	$$
	Since
	$$
	a_3 = a_1 - (m-1)\kappa^2
	$$
	we have $a_1 \geq a_3 > 0$ by condition \eqref{par_2}. We claim that 
	$$
		\frac{|D u(x_0)|^2}{W^2(x_0)}=\|\nabla u(x_0)\|^2 \leq \frac{a_2}{a_3} \, ,
	$$
	that is,
	\begin{equation}\label{first-estimate}
		W(x_0)\leq\sqrt[]{\frac{a_3}{a_3-a_2}} \, .
	\end{equation}
Indeed, assume by contradiction that
	\begin{equation} \label{abs-estimate}
		\|\nabla u(x_0)\|^2 > \frac{a_2}{a_3}.
	\end{equation}
	Then, from \eqref{zeroth-estimate} it follows
	$$
		\frac{\Delta_g\eta}{\eta+\delta} \geq a_3\|\nabla u\|^2 - a_2 > 0 \qquad \text{at } \, x_0 \, ,
	$$
	hence $\Delta_g\eta > 0$ and, by \eqref{main-inequality} and \eqref{zeroth-estimate} again,
	$$
		(m-1)\kappa^2\|\nabla u\|^2 \geq \frac{\Delta_g\eta}{\eta} \geq \frac{\Delta_g\eta}{\eta+\delta} \geq a_1\|\nabla u\|^2 - a_2 \qquad \text{at } \, x_0 \, ,
	$$
	leading to
	$$
		a_2 \geq (a_1 - (m-1)\kappa^2)\|\nabla u\|^2 = a_3\|\nabla u\|^2  \qquad \text{at } \, x_0 \, ,
	$$
	which contradicts \eqref{abs-estimate} and proves our claim. 
	
	On the other hand, if $|Du(x_0)|\leq a_0\gamma^\ast$, then 
	\begin{equation}\label{second-estimate}
		W(x_0)\leq\sqrt[]{1+a_0^2(\gamma^\ast)^2}
	\end{equation}
	Since $x_0$ is a global maximum point for $z$ in $\Omega$, we have $z(x) \leq z(x_0)$, that is,
	$$
	W(x) \leq W(x_0) \frac{\eta(x_0)}{\eta(x)}
	$$
	for every $x\in B_{R_1} \subseteq \Omega$. Note that
	$$
	\begin{array}{lcl}
	\disp \frac{\eta(x_0)}{\eta(x)} & = & \disp \frac{e^{LR(\sqrt{\eps^2 + 1} - \sqrt{\eps^2 + r(x_0)^2/R^2} - qu(x_0)/R)} - 1}{e^{LR(\sqrt{\eps^2 + 1} - \sqrt{\eps^2 + r(x)^2/R^2} - q\gamma(x))} - 1} \\[0.4cm]
	& \leq & \disp \frac{e^{LR(\sqrt{\eps^2 + 1}-\eps)} - 1}{e^{LR(\sqrt{\eps^2 + 1} - \sqrt{\eps^2 + r(x)^2/R^2} - q\gamma(x))} - 1}
	\end{array}
	$$
	hence
	$$
	W(x) \leq W(x_0) \left( \frac{e^{LR(\sqrt{\eps^2 + 1} - \eps)} - 1}{e^{LR(\sqrt{\eps^2 + 1} - \sqrt{\eps^2 + r(x)^2/R^2} - q\gamma(x))} - 1} \right) \, .
	$$
	The latter, together with \eqref{first-estimate} and \eqref{second-estimate}, implies the desired estimate.
	%
	%
\end{proof}

\section{Uniformly elliptic operators on manifolds with $\Ricc \geq 0$} \label{sec_L}

Having shown that an entire minimal graph with at most linear growth on one side has globally bounded gradient, we need to show \eqref{Dv} and \eqref{Hess_v}. We shall prove both of them under the only conditions
	\begin{equation}\label{eq_persezione}
	\Ricc \ge 0, \qquad |Du| \in L^\infty(M).
	\end{equation}
In such generality, it seems difficult to apply the ``elliptic" approach in \cite{ccm}, adapted in \cite{dingjostxin,ding_new}. To justify the statement, we observe that the method in \cite{ccm} relies on the construction of a function $\varrho$ satisfying
	\begin{equation}\label{eq_varrho}
	C^{-1} r \le \varrho \le Cr, \qquad |D\varrho| \le C, \qquad \Delta_g \varrho \le \frac{C}{\varrho},
	\end{equation}
for some absolute constant $C$. When considering harmonic functions, the third condition is replaced by $\Delta \varrho \le C/\varrho$, thus by comparison theory the choice $\varrho = r$ is admissible. On the contrary, to our knowledge, for minimal graphs the existence of $\varrho$ satisfying \eqref{eq_varrho} is currently unknown under the sole assumptions \eqref{eq_persezione}. If $M$ has Euclidean volume growth, we mention that in \cite{dingjostxin, ding_new} the authors used as $\varrho$ a reparametrization of the Green kernel of the Laplacian on $M$. Although the inequality $\Delta_g \varrho \le C/\varrho$ may not hold pointwise, the integral estimates for $|D^2\varrho|$ provided in \cite{cm} suffice to estimate $\Delta_g \varrho$ and apply the method in \cite{ccm}, as done in \cite[Lemma 7.1]{ding_new}. However, to our knowledge, estimates like those in \cite{cm} are not yet (if ever) available on manifolds with $\Ricc \ge 0$ but whose volume growth is less than Euclidean.
 
For these reasons, inspired by \cite{liann,sal92} we choose a different approach via the heat equation. Throughout this section, let $(M,\sigma)$ be a complete Riemannian manifold of dimension $m\geq 2$ with $\Ricc\geq 0$. Let $L$ be the linear uniformly elliptic operator defined by
\begin{equation} \label{L_def}
	L\psi = \div(AD\psi)
\end{equation}
where $A$ is a measurable section of $T^{1,1}M$ satisfying
\begin{equation} \label{A_cond}
	\text{i)} \;\; \alpha^{-1}|X|^2 \leq \langle AX,X\rangle \quad \text{and} \quad \text{ii)} \;\; |AX| \leq \alpha|X| \qquad \forall X\in TM
\end{equation}
for some constant $\alpha>0$. Hereafter, we shall assume that $A$ is smooth, the general case being obtainable by approximation.

We denote by $H_L(x,y,t)$ the minimal heat kernel associated to the parabolic operator $\partial_t - L$, that is, the unique continuous function on $M\times M\times\RR^+$ such that for every $\psi\in C^\infty_0(M)$ the function $u$ defined by
$$
	u(t,x) = \int_M H_L(x,y,t) \psi(y) \, \di y \qquad \forall \, (t,x) \in \RR^+\times M
$$
is a solution to 
\begin{equation} \label{h_eq}
	\partial_t u = L u
\end{equation}
on $\RR^+\times M$ satisfying
\begin{itemize}
	\item [i)] $u(t,\,\cdot\,) \to \psi$ pointwise on $M$ as $t\searrow 0$,
	\item [ii)] $u \leq v$ on $(0,T) \times M$ for every $v \in C^2([0,T)\times M)$, $T>0$, such that
	$$
		\begin{cases}
			\partial_t v = L v & \text{on } \, (0,T)\times M \, , \\
			\psi \leq v(0,\,\cdot\,) & \text{on } \, M \, .
		\end{cases}
	$$
\end{itemize}
If the endomorphism $A$ is self-adjoint with respect to $\metric$, the minimal heat kernel $H_L$ is a symmetric function of the space variables, that is,
\begin{equation}
	H_L(x,y,t) = H_L(y,x,t) \qquad \forall \, x,y\in M, \, \forall \, t > 0 \, .
\end{equation}
By \cite{sal92}, see Corollary 6.2 and Theorem 6.3, there exist positive constants $C_i > 0$, $1\leq i\leq 6$, depending only on $m$ and $\alpha$ such that, for every $x,y\in M$ and $t>0$,
\begin{equation} \label{H_Gbound}
	C_1 \frac{\exp\left(-C_2\dfrac{\dist(x,y)^2}{t}\right)}{\sqrt{\left|B_{\sqrt{t}}(x)\right|\left|B_{\sqrt{t}}(y)\right|}} \leq H_L(x,y,t) \leq C_3 \frac{\exp\left(-C_4\dfrac{\dist(x,y)^2}{t}\right)}{\sqrt{\left|B_{\sqrt{t}}(x)\right|\left|B_{\sqrt{t}}(y)\right|}}
\end{equation}
and
\begin{equation} \label{dtH_Gbound}
	|\partial_t H_L(x,y,t)| \leq \frac{C_5}{t} \frac{\exp\left(-C_6\dfrac{\dist(x,y)^2}{t}\right)}{\sqrt{\left|B_{\sqrt{t}}(x)\right|\left|B_{\sqrt{t}}(y)\right|}}
\end{equation}

Remarks on \eqref{H_Gbound} will be given in the Appendix.
We first need the following simple estimate on the volume of geodesic balls.

\begin{lemma} \label{lem_balls}
	Let $(M^m,\sigma)$ be a complete manifold with $\Ricc \geq 0$. For every $x,y\in M$ and for every $R>0$ it holds 
	$$
		\left|B_R(x)\right| \left(1 + \frac{\dist(x,y)}{R}\right)^{-\frac{m}{2}} \leq \sqrt{\left|B_R(x)\right|\left|B_R(y)\right|} \leq \left|B_R(x)\right| \left(1 + \frac{\dist(x,y)}{R}\right)^{\frac{m}{2}}
	$$
\end{lemma}
\begin{proof}
By Bishop-Gromov's comparison Theorem we have
\[
\frac{|B_R(x)|}{|B_r(x)|}\leq \left( \frac{R}{r}\right)^m, \qquad 0<r \le R<\infty,
\]
thus 
\[
|B_R(y)|\leq|B_{R+\dist(x,y)}(x)|\leq|B_R(x)|\left(1+\frac{\dist(x,y)}{R}\right)^m,
\]
and the thesis follows.
\end{proof}

Next, we recall that $L$ generates a diffusion which is stochastically complete (cf. \cite{grigoryan}), that is, the following holds:

\begin{lemma} \label{lem_stoch}
	Let $M$ be a complete manifold with $\Ricc\geq 0$, and let $A,L$ be as in \eqref{L_def}-\eqref{A_cond}, with $A$ self-adjoint and smooth. Then
	$$
		\int_M H_L(x,y,t) \, \di y = 1 \qquad \forall \, (t,x) \in \RR^+ \times M \, .
	$$
\end{lemma}

The result is stated with no proof in the discussion following \cite[Theorem 7.4]{sal92}. We here provide an argument for the convenience of the reader.

\begin{proof}
Since $L$ is uniformly elliptic and $M$ has polynomial volume growth as a consequence of $\Ricc\geq0$, by Theorem 4.1 of \cite{amr} we have that for any $\lambda>0$ the only entire bounded solution $v$ of $Lv = \lambda v$ on $M$ is $v\equiv 0$. Then the conclusion follows by \cite[Theorem 3.11]{prsmemoirs}.
\end{proof}

With the above preparation, we are ready to state the following asymptotic mean value theorem. Our method is inspired by the one in \cite{liann}, where the author considered the case $L = \Delta$, but with \tcb{a difference to be stressed.} Indeed, in \cite{liann} the author \tcr{uses} the Li-Yau's differential Harnack inequality to get rid of a boundary term at infinity. The inequality holds for solutions of the heat equation, but in general it \tcr{may fail} for solutions of $\partial_t u = Lu$, unless one has a uniform control on the gradient of $A$ on the entire $M$, \tcr{see for instance \cite[p.~433]{sal92}}. As we will apply our results to $A = W \mathrm{Id} - W^{-1} \di u \otimes Du$, with $W = \sqrt{1+|Du|^2}$, 
in our setting only an $L^\infty$ control on $A$ is available. \tcb{One may therefore use De Giorgi-Nash-Moser's theory to get H\"older estimates in space for $u$, see \cite[Corollary 5.5]{sal92}, but these seem insufficient to treat the boundary term.} 

In view of the above, we shall modify the method in \cite{liann}. \tcb{The main idea here is the use of upper level sets of $H_L$ rather than geodesic balls.} Note that we do not assume a Euclidean volume growth. We start with the following

\begin{lemma}\label{lem_partialtu}
	Let $(M^m,\metric)$ be a complete, noncompact manifold with $\Ricc\geq0$ and let $A,L$ be as in \eqref{L_def}-\eqref{A_cond}, with $A$ self-adjoint and smooth. If $f \in C^2(M) \cap L^\infty(M)$ satisfies $Lf\leq 0$ on $M$ then the function $u : \RR^+ \times M \to \RR$ given by
	\begin{equation} \label{u_def}
		u(t,x) = \int_M f(y) H_L(x,y,t) \, \di y \qquad \forall \, (t,x) \in \RR^+\times M
	\end{equation}
	satisfies
	\begin{equation} \label{li_u_mon}
		\partial_t u \leq 0 \qquad \text{on } \, \RR^+\times M \, .
	\end{equation}
\end{lemma}

\begin{proof}
	Note that the integral on the RHS of \eqref{u_def} converges for every $(t,x)\in\RR^+\times M$ since $f\in L^\infty(M)$ and because of \eqref{H_Gbound} and Lemma \ref{lem_balls}. Also note that $H_L$ is smooth as a consequence of the regularity assumptions on $A$. By Lemma \ref{lem_stoch}, $u$ only varies by an additive constant if so does $f$, hence without loss of generality we can assume $\inf_M f = 0$. Let $(t,x) \in \RR^+ \times M$ be fixed. For notational convenience, for every $a>0$ we define
	\begin{equation}\label{def_varphia}
		\varphi^a(y) = H_L(x,y,t)-a \ \ \  \forall \, y \in M, \qquad \Omega_a = \{ y \in M : \varphi_a(y) > 0 \}.
	\end{equation}
	Because of \eqref{H_Gbound} it holds $H_L(x,y,t) \to 0$ as $y \to \infty$ in $M$, hence the collection $\{\Omega_a\}_{a>0}$ is an exhaustion of $M$ by relatively compact open subsets, with $\Omega_a \subseteq \Omega_b$ when $a \geq b$. By \eqref{dtH_Gbound} and boundedness of $f$, we can apply Lebesgue's dominated convergence theorem to get
	\begin{equation} \label{li_pp0}
		\partial_t u(t,x) = \int_M f(y) \partial_t H_L(x,y,t) \, \di y = \lim_{a\to 0^+} \int_{\Omega_a} f(y) \partial_t H_L(x,y,t) \, \di y.
	\end{equation}
	Therefore, since $Lf \le 0$ and $\varphi_a>0$ on $\Omega_a$, \eqref{li_u_mon} holds \tcr{by monotone convergence} if we prove the inequality
	\begin{equation}\label{eq-negat}
		\int_{\Omega_a} f(y) \partial_t H_L(x,y,t) \, \di y \leq \int_{\Omega_a} \varphi_a(y) Lf (y) \, \di y.
	\end{equation}
	Because of $\partial_t H_L(x,y,t) = L_y H_L(x,y,t) = L\varphi(y) = L\varphi_a(y)$, we have
	$$
		\int_{\Omega_a} f(y) \partial_t H_L(x,y,t) \, \di y = \int_{\Omega_a} f(y) L_y H_L(x,y,t) \, \di y = \int_{\Omega_a} f(y) L \varphi_a(y) \, \di y \, .
	$$
	Since $H_L\in C^\infty(\RR^+\times M)$, we have $\varphi \in C^\infty(M)$ and for almost every $a>0$ the set $\Omega_a$ has smooth boundary. Let $a>0$ be a regular value for $\varphi$. By Green's identity, since $\varphi_a = 0$ on $\partial\Omega_a$
	$$
		\int_{\Omega_a} f(y) L \varphi_a(t,y) \, \di y = \int_{\Omega_a} \varphi_a(t,y) Lf (y) \, \di y + \int_{\partial\Omega_a} f(y) \langle AD\varphi_a(y),\nu \rangle \, \di \mathcal H^{m-1}(y)
	$$
	where $\nu = -D\varphi_a/|D\varphi_a|$ is the outward pointing normal on $\partial\Omega_a$. Noting that $f \geq 0$, that $\varphi_a$ is non-increasing in the direction of $\nu$ and that $A$ is positive definite, we see that $f\langle AD\varphi_a,\nu\rangle \leq 0$ on $\partial\Omega_a$ and therefore the second integral is non-positive, which implies the desired inequality \eqref{eq-negat}.
\end{proof}

\begin{proposition}\label{prop_allaLi}
	Let $(M^m,\metric)$ be a complete, noncompact manifold with $\Ricc\geq0$ and let $A,L$ be as in \eqref{L_def}-\eqref{A_cond}, with $A$ self-adjoint and smooth. If $f\in C^2(M) \cap L^\infty(M)$ satisfies $Lf \leq 0$ on $M$, then for any $x\in M$
	\begin{align}
		\label{mean_value_f}
		& \lim_{R\to \infty} \frac{1}{|B_R(x)|} \int_{B_R(x)} f(y) \, \di y = \inf_M f \, .
	\end{align}
\end{proposition}

\begin{proof}
	Without loss of generality, we assume $\inf_M f = 0$. Let $u : \RR^+\times M \to \RR$ be the function defined by \eqref{u_def}. Note that $u$ is the minimal solution to  the parabolic equation $\partial_t u = Lu$ on $\RR^+\times M$ corresponding to the initial datum $u(0^+,\,\cdot\,) = f$. Hence, by the maximum principle and the monotonicity \eqref{li_u_mon} we have
	$$
		\inf_M f \leq u(t,x) \leq f(x) \qquad \forall \, (t,x) \in \RR^+\times M \, .
	$$
	In particular, the limit
	$$
		u_\infty(x) = \lim_{t\to\infty} u(t,x)
	$$
	is well defined for every $x\in M$. The convergence $u(t,\,\cdot\,) \to u_\infty$ is uniform on compact subsets, $u_\infty$ is bounded and $Lu_\infty = 0$. Since $M$ is complete and has nonnegative Ricci curvature, the operator $L$ enjoys a Liouville property, see Theorem 7.4 of \cite{sal92}. In particular, $u_\infty$ must be constant. Since $\inf_M f \leq u_\infty \leq f$, it must be $u_\infty\equiv\inf_M f = 0$, that is,
	\begin{equation} \label{li_u_lim}
		\lim_{t\to\infty} u(t,x) = 0 \qquad \forall \, x \in M \, .
	\end{equation}
	To conclude the proof of \eqref{mean_value_f}, we observe that
	\begin{align*}
		u(t,x) & = \int_M H_L(x,y,t) f(y) \, \di y \\
		& \geq \frac{C_1}{|B_{\sqrt{t}}(x)|} \int_M \left(1 + \frac{\dist(x,y)}{\sqrt{t}}\right)^{-m/2} \exp\left(-C_2\dfrac{\dist(x,y)^2}{t}\right) f(y) \, \di y \\
		& = \frac{C_1}{|B_{\sqrt{t}}(x)|} \int_0^{\infty} \left(1 + \frac{r}{\sqrt{t}}\right)^{-m/2} \exp\left(-C_2\dfrac{r^2}{t}\right) \int_{\partial B_r(x)} f(y) \, \di \mathcal H^{m-1}(y) \, \di r \\
		& \geq \frac{C_1}{|B_{\sqrt{t}}(x)|} \int_0^{\sqrt{t}} \left(1 + \frac{r}{\sqrt{t}}\right)^{-m/2} \exp\left(-C_2\dfrac{r^2}{t}\right) \int_{\partial B_r(x)} f(y) \, \di \mathcal H^{m-1}(y) \, \di r \\
		& \geq \frac{2^{-m/2} e^{-C_2} C_1}{|B_{\sqrt{t}}(x)|} \int_0^{\sqrt{t}}  \int_{\partial B_r(x)} f(y) \, \di \mathcal H^{m-1}(y) \, \di r \\
		& = \frac{2^{-m/2} e^{-C_2} C_1}{|B_{\sqrt{t}}(x)|} \int_{B_{\sqrt{t}}(x)} f(y) \, \di y \, .
	\end{align*}
	Since $f\geq 0$, by comparison we have
	$$
		\lim_{t\to\infty} \frac{1}{|B_{\sqrt{t}}(x)|}\int_{B_{\sqrt{t}}(x)} f(y) \, \di y = 0 = \inf_M f
	$$
	as desired.
\end{proof}

From the above result, we also obtain information on the spherical mean of $u$. This follows from the next variant of de L'H\^opital's Theorem.

\begin{lemma}
	Let $h,g \in L^\infty_\loc(\RR^+)$ satisfy $h\geq 0$, $g>0$ a.e. and $g\not\in L^1(\infty)$. Then, 
	\begin{equation}\label{eq_Hop}
	\mathrm{ess} \liminf_{r \to \infty} \frac{h(r)}{g(r)} \le \liminf_{r\to \infty} \frac{\displaystyle\int_0^r h(t) \, \di t}{\displaystyle\int_0^r g(t) \, \di t}.
	\end{equation}
\end{lemma}

\begin{proof}
Denote by $A$ and $B$, respectively, the left-hand side and right-hand side of \eqref{eq_Hop}. For $A' < A$, fix $R_0$ such that $h \ge A' g$ a.e. on $(R_0,\infty)$. Then, for each $r > R_0$,  
	$$
	\frac{\displaystyle\int_0^r h(t) \, \di t}{\displaystyle\int_0^r g(t) \, \di t} = \frac{\displaystyle\int_0^{R_0} h(t) \, \di t + \int_{R_0}^r h(t) \, \di t}{\displaystyle\int_0^{R_0} g(t) \, \di t + \int_{R_0}^r g(t) \, \di t} \geq \frac{\displaystyle\int_0^{R_0} h(t) \, \di t + A' \int_{R_0}^r g(t) \, \di t}{\displaystyle\int_0^{R_0} g(t) \, \di t + \int_{R_0}^r g(t) \, \di t} \, .
	$$
Since $g\not\in L^1(\infty)$, letting $r\to \infty$ along a sequence realizing $B$ we get $B \ge A'$, and the thesis follows by letting $A' \uparrow A$. 
\end{proof}

\begin{corollary} \label{cor_mean_value}
	Let $(M^m,\metric)$ be a complete (connected) Riemannian manifold with infinite volume. Let $0 \leq f\in L^1_\loc(M)$ and $x\in M$ and assume that
	$$
	\liminf_{R\to \infty} \frac{1}{|B_R(x)|} \int_{B_R(x)} f(y) \, \di y = \inf_M f \, .
	$$
	Then
	$$
	\mathrm{ess}\liminf_{R\to \infty} \frac{1}{|\partial B_R(x)|} \int_{\partial B_R(x)} f(y) \, \di \mathcal H^{m-1}(y) \, \di y = \inf_M f \, .
	$$
\end{corollary}

\begin{proof}
	The functions $h$ and $g$ defined by
	$$
	h(t) = \int_{\partial B_t(x)} f(y) \, \di \mathcal H^{m-1}(y) \quad \text{and} \quad g(t) = |\partial B_t(x)| \, \qquad \forall \, t > 0
	$$
	satisfy the assumptions of the previous Lemma (note that $1/g \in L^\infty_\loc(\R^+)$ by \cite[Prop. 1.6]{bmr_mem}, since $M$ is non-compact). The thesis follows from the next chain of inequalities:
	\[
	\begin{array}{lcl}
	\disp \inf_M f & \le & \disp \mathrm{ess}\liminf_{R\to \infty} \frac{1}{|\partial B_R(x)|} \int_{\partial B_R(x)} f(y) \, \di \mathcal H^{m-1}(y) \, \di y \\[0.4cm]
	& \le & \disp \liminf_{R \to \infty} \frac{1}{|B_R(x)|} \int_{B_R(x)} f(y) \, \di y \le \inf_M f.
	\end{array}
	\]
\end{proof}

We are ready to state our second main result of the section, which will enable us to prove the Hessian estimate \eqref{Hess_v}. The argument below seems to be new.

\begin{theorem}\label{teo_mvLf}
	Let $(M^m,\metric)$ be a complete manifold with $\Ricc\geq0$ and let $A,L$ be as in \eqref{L_def}-\eqref{A_cond}, with $A$ self-adjoint and smooth. If $f\in L^\infty(M)$ satisfies $Lf \leq 0$ on $M$, then for any $x\in M$
	\begin{align}
		\label{mean_value_Lf}
		& \lim_{R\to \infty} \frac{R^2}{|B_R(x)|} \int_{B_R(x)} Lf(y) \, \di y = 0 \, .
	\end{align}
\end{theorem}

\begin{proof}
	Without loss of generality, we assume $\inf_M f = 0$. Fix $x \in M$. We refer to the proof of Lemma \ref{lem_partialtu} for notation, and in particular, for $t>0$ and $a>0$ we define $\varphi_a(y)$ and $\Omega_a$ as in \eqref{def_varphia}. As already observed, $\{\Omega_a\}$ is an exhaustion of $M$, increasing as $a$ decreases. Furthermore, for almost every $a>0$ the boundary $\partial \Omega_a$ is smooth. From the proof of Lemma \ref{lem_partialtu} we get
	\begin{equation}\label{eq_derHeatL}
	\int_{\Omega_a}	f(y) \partial_t H_L(x,y,t) \, \di y \le \int_{\Omega_a} \varphi_a(y) Lf (y) \, \di y.
	\end{equation}	
On the other hand, since $f\geq0$, by \eqref{dtH_Gbound} and Lemma \ref{lem_balls} we can estimate
	\begin{align*}
		\int_{\Omega_a} f(y) & \partial_t H_L(x,y,t) \, \di y \\
		& \geq - \frac{C_5}{t}\frac{1}{|B_{\sqrt{t}}(x)|} \int_{\Omega_a} f(y) \left(1 + \frac{\dist(x,y)}{\sqrt{t}}\right)^{m/2} \exp\left(-C_6\frac{\dist(x,y)^2}{t}\right) \, \di y \, .
	\end{align*}
	By \eqref{H_Gbound} and Lemma \ref{lem_balls} we also have the bounds
	\begin{align*}
		\frac{C_1}{|B_{\sqrt{t}}(x)|} & \left(1+\frac{\dist(x,y)}{\sqrt{t}}\right)^{-m/2} \exp\left(-C_2\frac{\dist(x,y)^2}{t}\right) \\
		& \leq H_L(x,y,t) \leq \frac{C_3}{|B_{\sqrt{t}}(x)|} \left(1+\frac{\dist(x,y)}{\sqrt{t}}\right)^{m/2} \exp\left(-C_4\frac{\dist(x,y)^2}{t}\right) \, .
	\end{align*}
	Now, fix $k>1$ large enough so that
	$$
		C_3 (1+s)^{m/2} e^{-C_4 s^2} \leq \frac{1}{2} C_1 2^{-m/2} e^{-C_2} \qquad \forall s \geq k
	$$
	and pick
	$$
		a = \frac{C_1 2^{-m/2} e^{-C_2}}{2|B_{\sqrt{t}}(x)|} \, .
	$$
	With this choice, we have
	$$
		\begin{cases}
			\varphi \leq a & \text{on } \, M \setminus B_{k\sqrt{t}}(x) \, , \\
			\varphi \geq 2a & \text{on } \, B_{\sqrt{t}}(x) \, ,
		\end{cases}
	$$
	hence $B_{\sqrt{t}}(x) \subseteq \Omega_a \subseteq B_{k\sqrt{t}}(x)$ and $\varphi_a \geq a$ on $B_{\sqrt{t}}(x)$. Thus, using also \eqref{eq_derHeatL} we can estimate
	\begin{align*}
		0 & \geq \frac{C_1 2^{-m/2} e^{-C_2}}{2|B_{\sqrt{t}}(x)|} \int_{B_{\sqrt{t}}(x)} Lf(y) \, \di y = a \int_{B_{\sqrt{t}}(x)} Lf(y) \, \di y \\
		& \geq \int_{B_{\sqrt{t}}(x)} \varphi_a(y)Lf(y) \, \di y \geq \int_{\Omega_a} \varphi_a(y)Lf(y) \, \di y \\
		& \geq \int_{\Omega_a} f(y)\partial_t H_L(x,y,t) \, \di y \\
		& \geq - \frac{C_5}{t|B_{\sqrt{t}}(x)|} \int_{\Omega_a} f(y) \left(1 + \frac{\dist(x,y)}{\sqrt{t}}\right)^{m/2} \exp\left(-C_6\frac{\dist(x,y)^2}{t}\right) \, \di y \\
		& \geq - \frac{C_7}{t|B_{\sqrt{t}}(x)|} \int_{B_{k\sqrt{t}}(x)} f(y) \, \di y
	\end{align*}
	where
	$$
		C_7 = C_5 \sup \left\{ (1+s)^{m/2} e^{-C_6 s^2} : s > 0 \right\} < \infty \, .
	$$
	Summing up, there exists a constant $C>0$, depending only on $C_i$, $1\leq i\leq 7$, such that
	$$
		0 \geq \frac{t}{|B_{\sqrt{t}}(x)|} \int_{B_{\sqrt{t}}(x)} Lf(y) \, \di y \geq - \frac{C}{|B_{\sqrt{t}}(x)|} \int_{B_{k\sqrt{t}}(x)} f(y) \, \di y \, .
	$$
	Since $f\geq0$, by Bishop-Gromov theorem we also have
	$$
		0 \geq \frac{t}{|B_{\sqrt{t}}(x)|} \int_{B_{\sqrt{t}}(x)} Lf(y) \, \di y \geq - \frac{Ck^m}{|B_{k\sqrt{t}}(x)|} \int_{B_{k\sqrt{t}}(x)} f(y) \, \di y \, .
	$$
	By Proposition \ref{prop_allaLi} we have that the RHS of this inequality converges to $\inf_M f = 0$ as $t\to\infty$, and the conclusion follows.
\end{proof}

\section{Proof of Theorem \ref{teo_main}, $(ii)$}

Combining Corollary \ref{cor_boundedgrad}, Proposition \ref{prop_allaLi} and Theorem \ref{teo_mvLf}, we get

\begin{proposition}\label{prop_asi_grad_hess}
	Let $(M^m,\sigma)$ be a complete Riemannian manifold with $\Ricc \geq 0$ and 
	\[
	\Ricc^{(\ell)}(\nabla r) \ge - \frac{\bar \kappa^2}{1+r^2}, \qquad \ell = \max\{1,m-2\},
	\]
for some $\bar \kappa \in \R^+_0$ and where $r$ is the distance from a fixed origin. Let $u \in C^\infty(M)$ be a non-constant solution to  \eqref{P} which grows at most linearly on one side. Then, for each $x \in M$,
	\begin{align}
		\label{Du_L2}
		& \lim_{R\to\infty} \frac{1}{|B_R(x)|} \int_{B_R(x)} |Du|^2 \di x = \sup_M |Du|^2, \\
		\label{Hess_u_L2}
		& \lim_{R\to\infty} \frac{R^2}{|B_R(x)|} \int_{B_R(x)} |D^2 u|^2 \di x = 0.
	\end{align}
\end{proposition}

\begin{proof}
Because of Corollary \ref{cor_boundedgrad}, in our assumptions $|Du| \in L^\infty(M)$, hence by \eqref{eq_Lapla} the operator 
	\[
	L\phi \doteq W\Delta_g\phi = \diver \big( W g^{ij} \phi_i\partial_{x_j} \big)
	\]
is uniformly elliptic on $M$. By the Jacobi equation, $f = 1/W$ is a non-negative solution to  $L f \le - \|\II\|^2 \le 0$, and therefore $-W^2 \in L^\infty(M)$ satisfies $L (-W^2) \le 0$. Applying Proposition \ref{prop_allaLi} to $-W^2$ and Theorem \ref{teo_mvLf} to $f$ we deduce 	\begin{align}
	\label{W_mean}
	& \lim_{R\to\infty} \frac{1}{|B_R(x)|} \int_{B_R(x)} W^2 \di x = \sup_M W^2 \\
	\label{second_hess} & \limsup_{R\to\infty} \frac{R^2}{|B_R(x)|} \int_{B_R(x)} \|\II\|^2 \di x \le - \lim_{R\to\infty} \frac{R^2}{|B_R(x)|} \int_{B_R(x)} L f \di x = 0.
	\end{align}
From \eqref{W_mean} we readily deduce \eqref{Du_L2}. On the other hand, note that
	\[
	\begin{array}{lcl}
	\|\II\|^2 & = & W^{-2}g^{ik}u_{kj}g^{jl}u_{li} \\[0.2cm]
	& = & W^{-2} \left\{ |D^2 u|^2 - 2 \left|D^2 u\left(\frac{Du}{W},\cdot\right)\right|^2 + \left[D^2 u\left(	\frac{Du}{W},\frac{Du}{W}\right)\right]^2 \right\}
	\end{array}
	\]
If $\di u(x) = 0$, then $\|\II\|^2 \ge W^{-2}|D^2 u|^2$. Otherwise, let $e_1 = Du/|Du|$ and choose a local orthonormal frame $\{e_\alpha\}$ for $e_1^\perp$ around $x$, where $2 \le \alpha \le m$. Then,    
	\[
	\begin{array}{lcl}
|D^2 u|^2 - 2 \left|D^2 u\left(\frac{Du}{W},\cdot\right)\right|^2 + \left[D^2 u\left(	\frac{Du}{W},\frac{Du}{W}\right)\right]^2 \\[0.3cm]	
	\sum_{\alpha,\beta} u_{\alpha\beta}^2 + 2 \sum_\alpha u^2_{1\alpha} + u_{11}^2 - 2 \frac{W^2-1}{W^2}\sum_j u_{1j}^2 + \frac{(W^2-1)^2}{W^4} u_{11}^2 \\[0.3cm]
	= \sum_{\alpha,\beta} u_{\alpha\beta}^2 + \frac{2}{W^2} \sum_\alpha u^2_{1\alpha} + \frac{1}{W^4} u_{11}^2
	\ge W^{-4} |D^2 u|^2
	\end{array}
	\]
Summarizing, we have $\|\II\|^2 \ge W^{-6}|D^2 u|^2$, thus from the boundedness of $W$ and from \eqref{second_hess} we conclude \eqref{Hess_u_L2}.
\end{proof}

We now conclude the proof of Theorem \ref{teo_main} with a blow-down procedure, for which we use some basic convergence results in the theory of limit spaces and nonsmooth spaces with Ricci curvature bounded below. All the tools needed herein can be found in \cite{H,AH,AH_2}.\par
Fix $o \in M$, and write $B_R = B_R(o)$. Because of Corollary \ref{cor_boundedgrad} and Proposition \ref{prop_asi_grad_hess},
	\begin{align}
		\label{Du_L2_2}
		& \lim_{R\to\infty} \frac{1}{|B_R|} \int_{B_R} |Du|^2 \di x = \sup_M |Du|^2, \\
		\label{Hess_u_L2_2}
		& \lim_{R\to\infty} \frac{R^2}{|B_R|} \int_{B_R} |D^2 u|^2 \di x = 0.
	\end{align}
Consider a tangent cone at infinity $X_\infty$ for $M$ based at $o$. By statement (2.1) in \cite{dphg}, the limit space $X_\infty$ also supports a Borel measure $\mes_\infty$ such that, up to a subsequence,
	\begin{equation}\label{eq_conver}
	\left( M, \lambda_n^{-1} \dist_\sigma, \lambda_n^{-m} \di x, o\right)  \stackrel{\pmgh}{\longrightarrow}(M_\infty,\di_\infty,\mes_\infty,o_\infty)
	\end{equation}
in the pointed-measured-Gromov-Hausdorff (pmGH) sense. For the precise definition of pGH and pmGH convergence we refer to \cite{gms}.
Here, $\{\lambda_n\} \subset \R^+$, $\lambda_n \to \infty$ as $n \to \infty$, and $\lambda_n^{-1}\dist_\sigma$ is the distance function induced by the rescaled metric $\sigma_n \doteq \lambda_n^{-2}\sigma$. Denote with $D_n$ and $\di x_n$ the induced connection and volume measure, and $B_R^{n}$ the metric balls centered at $o$ in $(M,\sigma_n)$. Therefore, $B_R^{n} = B_{\lambda_n R}$. Define $u_n = u/\lambda_n$. Then,  
	\begin{equation}\label{eq_un}
	|D_n u_n|_{\sigma_n} = |Du|, \qquad |D^2_n u_n|_{\sigma_n} = \lambda_n |D^2 u|
	\end{equation}
and therefore, by Arzel\'a-Ascoli Theorem, up to subsequences $u_n \to u_\infty \in \lip(M_\infty)$ locally uniformly, hence $u_n \to u_\infty$ strongly in $L^2$ on $B_R^\infty = B_R^{\di_\infty}(x_\infty)$, that is, 
	\[
	\begin{array}{l}
	\disp \lim_{n \to \infty} \int_{B_R^n} |u_n|^2 \di x_n = \int_{B_R^\infty} |u_\infty|^2 \di \mes_\infty, \\[0.5cm]
	\disp \lim_{n \to \infty} \int_{B_R^n} u_n \varphi \di x_n = \int_{B_R^\infty} u_\infty \varphi \di \mes_\infty
	\end{array}	
	\]
for each $\varphi$ bounded and continuous on a metric space $Z$ in which $(B_R^n, \di_n)$ and $(B_R^\infty, \di_\infty)$ are isometrically embedded and converge in Hausdorff sense, with $o_n \to o_\infty$ and $o_n$ the center of $B_R^n$.  
From 
	\[
	W_n \doteq \sqrt{1+ |D_nu_n|^2_{\sigma_n}} = \sqrt{1+|Du|^2} = W
	\]	
Scaling \eqref{Du_L2_2} and \eqref{Hess_u_L2_2} we therefore get, for each fixed $R>0$   
	\begin{align}
		\label{Du_L2_n}
		& \lim_{n\to\infty} \frac{1}{|B^{n}_R|_{\sigma_n}} \int_{B^{n}_R} |D_nu_n|_{\sigma_n}^2 \di x_{n} = \sup_M |Du|^2, \\
		\label{Hess_u_L2_n}
		& \lim_{n\to\infty} \frac{R^2}{|B^{n}_R|_{\sigma_n}} \int_{B^{n}_R} |D^2_n u_n|_{\sigma_n}^2 \di x_{n} = 0.
	\end{align}
In particular, from Newton's inequality $|\Delta_n u_n|^2 \le m|D_n^2 u_n|^2_{\sigma_n}$ and Bishop-Gromov's Theorem, $|B_R^n|_{\sigma_n} \le \omega_{m-1} R^m/m$ we deduce
	\begin{equation}\label{eq_Deltaun2}
	\int_{B^{n}_R} |\Delta_n u_n|^2 \di x_{n} \le \frac{\omega_{m-1}R^m}{|B^{n}_R|_{\sigma_n}} \int_{B^{n}_R} |D^2_n u_n|_{\sigma_n}^2 \di x_{n} \to 0 
	\end{equation}
as $n \to \infty$, and therefore
	\begin{equation}\label{eq_deltaunw}
	\begin{array}{lcl}
	\disp \int_{B^{n}_R} \varphi \Delta_n u_n \di x_{n} & \le & \left( \int_{B_R^n} \varphi^2 \di x_n\right)^{\frac{1}{2}} \left( \int_{B^{n}_R} |\Delta_n u_n|^2 \di x_{n} \right)^{\frac{1}{2}} \\[0.4cm]
	& \le & \max|\varphi| \left[\frac{\omega_{m-1}R^m}{m}\right]^{\frac{1}{2}} \left( \int_{B^{n}_R} |\Delta_n u_n|^2 \di x_{n} \right)^{\frac{1}{2}} \to 0.
	\end{array}
	\end{equation}
By \eqref{eq_Deltaun2} and \eqref{eq_deltaunw}, $\Delta_n u_n \to 0$ strongly in $L^2$. Combining $u_n \to u_\infty$ strongly in $L^2$ with 
	\[
	\sup_n \left( \int_{B_R^n} \big[ |u_n|^2 + |D_nu_n|^2_{\sigma_n} + (\Delta_n u_n)^2 \big] \di x_n \right) < \infty
	\] 
we infer by \cite[Thm. 4.4]{AH} that
	\begin{itemize}
	\item[$(i)$] $u_\infty \in \mathcal{D}(\Delta, B_R^\infty)$, the domain of the Laplacian on $B_R^\infty$;
	\item[$(ii)$] $\Delta_n u_n \to \Delta u_\infty$ on $B_R^n$ weakly in $L^2$, so in particular $\Delta u_\infty = 0$;
	\item[$(iii)$] $|D_nu_n|^2_{\sigma_n} \to |D_\infty u_\infty|^2_\infty$ in $L^1$-strongly in $B_r^n$, for each $r< R$; 
	\end{itemize}
In particular, setting $P \doteq \sup_M |Du|^2$, from \eqref{eq_un} and \eqref{Du_L2_n} we get
	\[
	\lim_{n \to \infty} \int_{B_R^n} \big| |D_nu_n|^2_{\sigma_n} - P \big| \di x_n \le \frac{\omega_{m-1}R^m}{|B_R^n|_{\sigma_n}} \int_{B_R^n} \Big( P - |D_nu_n|^2_{\sigma_n} \Big) \di x_n = 0
	\]
Using $(iii)$ and \cite[Prop. 1.27 (i)]{bps} (cf. also \cite{AH_2}), we therefore deduce $|D_nu_n|^2_{\sigma_n} - P \to |D_\infty u_\infty|^2_\infty -P$ strongly in $L^1$ on $B_r^\infty$ for each $r< R$, and thus
	\[
	0 = \lim_{n \to \infty} \int_{B_r^n} \big| |D_nu_n|^2_{\sigma_n} - P \big| \di x_n = \int_{B_r^\infty} \big| |D_\infty u_\infty|^2_{\infty} - P \big| \di \mes_\infty. 	
	\]
Concluding, $u_\infty$ solves
	\[
	\Delta u_\infty = 0, \qquad |D_\infty u_\infty|^2 = P \neq 0
	\]
on the $\RCD(0,m)$ space $(M_\infty, \di_\infty,\mes_\infty, x_\infty)$.  Bochner inequality (see \cite[Thm. 1.4]{H}) guarantees that $|D^2 u_\infty| \equiv 0$ on $M_\infty$. One concludes that $M_\infty = N \times \R$ by using \cite[Lem. 1.21]{abs}.

\section{Proof of Theorem \ref{teo_slower}}

If $M$ is parabolic, clearly the result follows from Theorem \ref{teo_main}. If $M$ is non-parabolic, the argument goes as in \cite[Theorem 3.6]{dingjostxin}, so we only sketch the main steps. In our assumptions, by Corollary \ref{cor_boundedgrad}, $|Du| \in L^\infty(M)$, hence $L = W \Delta_g$ is uniformly elliptic. The Harnack inequality in \cite{sal92} together with \eqref{eq_slower} imply that $|u(x)| = o(r(x))$ as $x$ diverges. By a standard cutoff argument using $Lu = 0$, the next Caccioppoli inequality holds: for each $\varphi\in \lip_c(M)$
	\begin{equation}\label{eq_caccio}
		\int_M \varphi^2 |Du|^2\di x \leq 4 \alpha^2 \int_M u^2|D\varphi|^2\di x.
	\end{equation}
In particular, having fixed $\eps>0$, by condition $u = o(r)$ we can also fix $R_0 = R_0(\eps) > 0$ such that for every $R \geq R_0$ we have $u^2 \leq \eps R^2$ on $B_{2R}$. Considering the Lipschitz cutoff function $\varphi$ which is $1$ on $B_{R}$, $0$ outside of $B_{2R}$ and satisfies $|D\varphi| \le 1/R$, we get
	\[
	\int_{B_R} |Du|^2 \di x \le \frac{4\alpha^2}{R^2} \int_{B_{2R}\backslash B_R} u^2 \di x \le \eps |B_{2R}| \le C \eps |B_R|
	\]
for every $R\geq R_0$, where we used the doubling property on $M$ coming from condition $\Ricc \ge 0$. From \eqref{Du_L2} we finally infer
	\[
	\sup_M |Du|^2 = \lim_{R \to \infty} \frac{1}{|B_R|} \int_{B_R} |Du|^2 \di x \le C\eps, 
	\]
and the thesis follows by letting $\eps \to 0$.

\section{Proof of Corollary \ref{cor_main}}

By Theorem \ref{teo_main}, $|Du| \in L^\infty(M)$ and any tangent cone at infinity of $M$ splits off a line. It is a general fact that, if $\Sec \ge 0$, a tangent cone splits if and only if $M$ itself splits. A proof of this result can be found in \cite[Thm. 4.6]{abfp}. Therefore, it remains to prove that $u$ only depends on the coordinate of a split line. Write $M = N^{m-1} \times \R$ with coordinates $(y_1,s_1)$, for some complete manifold $N^{m-1}$ with $\Sec \ge 0$, and consider the function $v_1 = \sigma(Du, \partial_{s_1})$, which by \eqref{killing-Du} satisfies $L v_1 = 0$ on $M$, where we set 
	\[
	L\phi \doteq W^{-1} \LL_W\phi = \diver \big( W^{-1} g^{ij} \phi_i \partial_{x_j}\big). 
	\] 
Our gradient estimate guarantees that $v_1$ is bounded and that $L$ is uniformly elliptic on $M$, and therefore, by \cite[Theorem 7.4]{sal92} we deduce that $v_1$ is constant on $M$. Hence,  
	\[
	u(y_1,s_1) = a_1s_1 + b_1 + u_2(y_1) \sqrt{1+a_1^2},
	\]
for some smooth function $u_2 : N^{m-1} \to \R$ and some $a_1,b_1 \in \R$. One easily checks that $u_2$ solves \eqref{P} on $N^{m-1}$. Since $u_2$ has at most linear growth on one side, and $N^{m-1}$ has non-negative sectional curvature, by the first part of the proof we deduce that either $u_2$ is constant or that $N^{m-1} = N^{m-2} \times \R$ and $u_2(y_2,s_2) = a_2 s_2 + b_2 + u_3(y_2)\sqrt{1+a_2^2}$. Iterating, we can write $M = N^{m-k} \times \R^k$ for some $k \in \{1,\ldots,m-2\}$ and for some complete manifold $N^{m-k}$ with $\Sec \ge 0$, and 
	\[
	u(z,(s_1,\ldots,s_k)) = \sum_{j=1}^k a_j s_j + b + u_{k+1}(z)\sqrt{1+a_k^2}	
	\]
for some $a_i,b \in \R$ and $u_{k+1} : N^{m-k} \to \R$. Indeed, we can continue the iteration procedure up until either $u_{k+1}$ is constant, or $k = m-2$ and $u_{m-1}$ is non-constant. In the latter case, observe that $N^2$ is a complete surface with $\Sec \ge 0$, hence $N^2$ is parabolic. Being $u_{m-1}$ non-constant, both $N^2$ and $u_{m-1}$ split as indicated in Theorem \ref{teo_main}, $(i)$. Summarizing, in each case we can conclude that $M = N^{m-k} \times \R^k$ for some $k \in \{1,\ldots, m-1\}$, and that 
	 \begin{equation}\label{eq_splitv}
	u(z,(s_1,\ldots,s_k)) = \sum_{j=1}^k a_j s_j + b	
	\end{equation}	 
for some $b \in \R$, as required. It is therefore sufficient to consider the splitting $\R^k = \R^{k-1} \times \R$ along a line in direction $(a_1,\ldots, a_k)$ to get the desired splitting $M = N \times \R$ of $M$ in such a way that $u(y,s) = as + b$. \par

\section{Proof of Proposition \ref{teo_counter}}

The following example is essentially that in \cite[p. 913]{kasuewashio}. Let $m \ge 4$. We consider a manifold $(P^{m-2},h)$ and smooth functions $f,\eta \in C^\infty(\R^+)$ to be chosen later, and define the following metric on $M \doteq \R \times \R^+ \times P$:
	\[
	\sigma = f(r)^2 \di t^2 + \di r^2 + \eta(r)^2h.
	\]
To compute the curvatures of $M$, we use the index agreement $1 \le a,b,c,l \le m$, $3 \le \alpha,\beta,\gamma,\delta \le m$. Let $\{\theta^\alpha\}$ be a local orthonormal coframe on $P$, with associated connection forms $\omega^\alpha_\beta$ obeying the structure equations
	\[
	\left\{ \begin{array}{l}
	\di \theta^\alpha = - \omega^\alpha_\beta \wedge \theta^\beta \\[0.1cm]
	\omega^\alpha_\beta = - \omega^\beta_\alpha 
	\end{array}\right.
	\]
and related curvature forms $\Theta^\alpha_\beta = \di \omega^\alpha_\beta + \omega^\alpha_\gamma \wedge \omega^\gamma_\beta$. Then, a local orthonormal coframe $\{\bar \theta^a\}$ on $M$ is given by 
	\[
	\bar \theta^1 = f \di t, \quad \bar \theta^2 = \di r, \quad \bar \theta^\alpha = \eta \theta^\alpha,
	\]
where, as usual, pull-backs to $M$ via the canonical projections onto $\R,\R^+$ and $P$ are implicit. Differentiating, one checks that the forms
	\[
	\bar \omega^\alpha_1 = 0, \quad \bar \omega^\alpha_2 = \frac{\eta'}{\eta} \bar \theta^\alpha, \quad \bar \omega^\alpha_\beta = \omega^\alpha_\beta, \quad \bar \omega^2_1 = - \frac{f'}{f} \bar \theta^1.
	\]
satisfy the structure equations on $M$ for the coframe $\{\bar \theta^a\}$, hence they are the connection forms of $\{\bar \theta^a\}$. The associated curvature forms $\bar \Theta^a_b = \di \bar \omega^a_b + \bar \omega^a_c \wedge \bar \omega^c_b$ are therefore
	\begin{equation}\label{eq_curvforms}
	\begin{array}{ll}
	\disp \bar \Theta^\alpha_1 = - \frac{\eta'f'}{\eta f} \bar \theta^\alpha \wedge \bar \theta^1, & \quad \disp \bar \Theta^\alpha_2 = \frac{\eta''}{\eta} \bar \theta^2 \wedge \bar \theta^\alpha, \\[0.3cm]
	\disp \bar \Theta^\alpha_\beta = \Theta^\alpha_\beta - \left( \frac{\eta'}{\eta} \right)^2 \bar \theta^\alpha \wedge \bar \theta^\beta & \quad \disp \bar \Theta^2_1 = - \frac{f''}{f} \bar \theta^2 \wedge \bar \theta^1, 
	\end{array}
	\end{equation}
The components $R^\alpha_{\beta \gamma \delta}$ and $\bar R^a_{bcl}$ of the $(3,1)$ curvature tensors of, respectively, $P$ and $M$, are given by the identities
		\[
		\Theta^\alpha_\beta = \frac{1}{2} R^\alpha_{\beta \gamma \delta} \theta^\gamma \wedge \theta^\delta, \qquad \bar \Theta^a_b = \frac{1}{2} \bar R^a_{bcl} \bar\theta^c \wedge \bar \theta^l,
		\]	
and thus, from \eqref{eq_curvforms}, we deduce
		\begin{equation}\label{curvature-operator}
		\begin{array}{c}
		0 = \bar R^2_{12\alpha} = \bar R^2_{1\alpha 1} = \bar R^2_{1\alpha\beta} = \bar R^\alpha_{12\beta} = \bar R^\alpha_{1\gamma \delta} = \bar R^\alpha_{2\gamma \delta}\\[0.3cm]
		\bar R^2_{121} = - \frac{f''}{f}, \quad \bar R^\alpha_{1\beta 1} = - \frac{\eta'f'}{\eta f} \delta^\alpha_\beta, \quad \bar R^\alpha_{2\beta 2} = -\frac{\eta''}{\eta} \delta^\alpha_\beta, \\[0.3cm]
		\bar R^\alpha_{\beta\gamma\delta} = \frac{1}{\eta^2} R^\alpha_{\beta\gamma\delta} - \left(\frac{\eta'}{\eta}\right)^2 \left[ \delta^\alpha_\gamma \delta_{\beta \delta} - \delta^\alpha_\delta \delta_{\beta \gamma} \right].
		\end{array}
		\end{equation}	
Assume that $(P,h)$ is the round sphere with curvature $1$, and let $\{e_\alpha\}$ and $\{\bar e_a\}$ be, respectively, the dual frames of $\{\theta^\alpha\}$ and $\{\bar \theta^a\}$. 
From \eqref{curvature-operator} we deduce that the curvature operator is diagonalized by the simple planes $\{\bar e_a \wedge \bar e_b\}$, so for $m \ge 4$ we get
\[
|\overline{\Sec}(\pi)| \le \max \left\{ \left|\frac{f''}{f}\right|, \left| \frac{\eta'f'}{\eta f}\right|, \left|\frac{1- (\eta')^2}{\eta^2}\right|,\left|\frac{\eta''}{\eta}\right|\right\}.
\]
In \cite{kasuewashio}, the authors chose the following functions $f,\eta$: given $\alpha, \beta \in (0,1)$ such that $m-1-\beta > 2+\alpha$, let $0 < \zeta_1,\zeta_2 \in C^\infty(\R^+)$ satisfy 
	\[
	\zeta_1(t) = \left\{ \begin{array}{ll}
	t & \text{ if } t \in (0,1] \\
	t^{-1-\alpha} & \text{ if } \, t \in [2,\infty), 
	\end{array}\right. \qquad  \zeta_2(t) = \int_t^\infty \zeta_1(s) \di s. 
	\]
Then, for $b,c \in \R^+$ they defined
	\[
	\eta(r) = \frac{1}{2} r + \frac{1}{2\zeta_2(0)} \int_0^r \zeta_2(s) \di s, \qquad f(r) = (b+r^2)^{\frac{\beta + 3 - m}{2}} + c.
	\]		
Note that with such a choice the metric extends in a $C^2$ way at $r=0$, giving rise to a complete manifold. Since the curvature operator is diagonalized by $\{\bar e_a \wedge \bar e_b\}$, 
\begin{equation}\label{eq_2lin}
\begin{array}{lcl}
\disp \overline{\Ric}^{(2)} & \geq & \disp \min\left\{-\frac{f''}{f}+\frac{1-(\eta')^2}{\eta^2},-\frac{f''}{f}-\frac{\eta'f'}{\eta f},-\frac{f''}{f}-\frac{\eta''}{\eta}, \right. \\[0.4cm] 
& & \disp \left.  \frac{1-(\eta')^2}{\eta^2} -\frac{\eta'f'}{\eta f}, \frac{1-(\eta')^2}{\eta^2} -\frac{\eta''}{\eta}, -\frac{\eta''}{\eta} -\frac{\eta'f'}{\eta f}, -2\frac{\eta'f'}{\eta f}, \right. \\[0.4cm]
& & \disp \left. + 2\frac{1-(\eta')^2}{\eta^2}, - 2\frac{\eta''}{\eta}\right\}.
\end{array}
\end{equation}
By the expression of $\eta,f$, the four terms in the second line of \eqref{eq_2lin} are positive, and it is easy to see that, when $b,c$ are large enough, the three terms in the first line are positive as well. the two terms in the third line are positive except at $r=0$. Whence, $\overline{\Ric}^{(2)} \ge 0$, and moreover $|\overline{\Sec}| \le \bar \kappa^2$ holds for a suitable $\bar \kappa>0$. Moreover, from the fact that $\overline{\Ricc}$ is diagonal in the basis $\{\bar e_a\}$ with
		\[
		\begin{array}{c}
		\overline{\Ricc}_{11} = - \frac{f''}{f} - (m-3)\frac{\eta'f'}{\eta f}, \qquad \overline{\Ricc}_{22} = - \frac{f''}{f} - (m-3)\frac{\eta''}{\eta}, \\[0.4cm]
		\overline{\Ricc}_{\alpha \beta} = \left[ - \frac{\eta'f'}{\eta f} - \frac{\eta''}{\eta} + (m-3)\frac{1- (\eta')^2}{\eta^2} \right]\delta_{\alpha\beta}
		\end{array}
		\]
we deduce that $\Ricc > 0$ if $b,c$ are chosen large enough. To construct linearly growing minimal graphs, consider a function $u : M \to \R$ of the coordinate $t$ alone. It follows that $\di u = u_a \bar \theta^a$ with $u_1 = (\partial_t u)/f$ and $u_a = 0$ for $a \ge 2$. The components of the Hessian $D^2 u$ obey the relation
		\[
		u_{ab} \bar \theta^b = \di u_a - u_c \bar \omega^c_a,
		\]
and from the expression of $\bar \omega^c_a$ we get
		\[
		\begin{array}{c}
		u_{11} = \frac{\partial^2_t u}{f^2}, \quad u_{21} = - \frac{f'}{f^2}\partial_t u, \\[0.3cm]
		\quad u_{1\alpha} = u_{22} = u_{2\alpha} = u_{\alpha\beta} = 0, 
		\end{array}
		\]
In particular, setting $W = \sqrt{1 + |Du|^2} = \sqrt{ 1 + u_1^2}$, 
		\[
		\diver\left( \frac{Du}{\sqrt{1+|Du|^2}} \right) = \frac{\Delta u}{W} - \frac{D^2 u(Du,Du)}{W^3}= \frac{\partial^2_t u}{f^2 W} - \frac{(\partial^2_t u)u_1^2}{f^2 W^3} = \frac{\partial^2_t u}{f^2W^3}.
		\]	
It follows that any affine function $u(t) = at + b$ gives rise to a minimal graph. Furthermore, $|Du| = a/f$ is bounded on $M$ since $f$ is bounded below by a positive constant, thus $u$ has at most linear growth.

\begin{appendix}

\section{}

Let $M$ be a complete Riemannian manifold with non-negative Ricci curvature, $\dim M = m$, and let $A,L,H_L$ be as in section \ref{sec_L}. In this Appendix, we discuss the two-sided bound in \eqref{H_Gbound} for $H_L$. While the upper bound is shown in \cite{sal92}, the argument for the lower bound is merely indicated with no proof. The approach relies on the following parabolic Harnack inequality in \cite[Corollary 5.4]{sal92}: given $p\in M$, $R>0$, $T>0$ and $\delta\in(0,1)$, if $u$ is a positive solution to  $\partial_t u = Lu$ on $B_R(p) \times (0,T)$, then
	\begin{equation} \label{diff_harn}
		\log\left( \frac{u(t,y)}{u(s,x)} \right) \leq C\left( \frac{\dist(x,y)^2}{s-t} + \left(\frac{1}{R^2} + \frac{1}{t} \right) (s-t) + 1 \right)
	\end{equation}
for every $x,y\in B_{\delta R}(p)$ and $0<t<s<T$, with $C = C(m,\delta,\alpha)>0$. A note of warning: in \cite[Corollary 5.4]{sal92}, the final $+1$ in brackets in \eqref{diff_harn} is missing. However, necessity of this correction becomes apparent by direct inspection of Moser's original proof, \cite[pages 110--112]{moser64}, in Euclidean setting (the analogue of \eqref{diff_harn} is \cite[Formula (1.5)]{moser64}). For the reader's convenience, we give a proof that the lower bound in \eqref{H_Gbound} follows from the upper one coupled with \eqref{diff_harn}, along the lines of the argument developed by Aronson and Serrin \cite{as67} in the Euclidean case. A few observations are in order.

First, in view of Lemma \ref{lem_balls} the upper bound in \eqref{H_Gbound} implies
\begin{equation} \label{G_upper}
	H_L(x,y,t) \leq \frac{C_3'}{|B_{\sqrt{t}}(x)|} \exp\left( - C_4'\dfrac{\dist(x,y)^2}{t} \right) \qquad \forall x,y\in M, \, t > 0
\end{equation}
with $C_3',C_4'>0$ depending only on $m$ and $\alpha$ (the ellipticity constant of $A$). Secondly, the differential Harnack inequality \eqref{diff_harn} applied to $u = H_L(x,\,\cdot\,,\,\cdot\,)$ yields
\begin{equation} \label{parab_H}
	H_L(x,y_1,t_1) \leq H_L(x,y_2,t_2) \exp\left( C\frac{\dist(y_1,y_2)^2}{t_2-t_1} + C\frac{t_2}{t_1} \right)
\end{equation}
for every $y_1,y_2\in M$ and $0 < t_1 < t_2 < \infty$, with $C = C(m,\alpha)>0$. Lastly, note that if we have the validity of a lower bound of the form
\begin{equation} \label{G_lower}
	H_L(x,y,t) \geq \frac{C_1'}{|B_{\sqrt{t}}(x)|} \exp\left( - C_2'\dfrac{\dist(x,y)^2}{t} \right) \qquad \forall x,y\in M, \, t > 0
\end{equation}
with $C_1',C_2'>0$ depending only on $m$ and $\alpha$, then, again by Lemma \ref{lem_balls}, a lower bound as that in \eqref{H_Gbound} holds for suitable constants $C_1 \in (0,C_1')$ and $C_2>C_2'$ depending only on $C_1'$, $C_2'$ and $m$. Hence, we limit ourselves to the proof that \eqref{G_lower} follows from \eqref{G_upper} and \eqref{parab_H} under the assumption $\Ricc\geq0$.

Fix a constant $c_0>2$ such that
\begin{equation} \label{c_0}
	\gamma \doteq m C_3' \int_{\sqrt{c_0}}^{+\infty} s^{m-1} e^{-C_4' s^2} \, \di s < 1 \, .
\end{equation}
Let $(x,y,t) \in M \times M \times \RR^+$ be given. By \eqref{parab_H} we have
\begin{equation} \label{HL_in1}
	H_L(x,x,t/2) \leq H_L(x,y,t) \exp\left( 2C \frac{\dist(x,y)^2}{t} + 2C \right)
\end{equation}
and also
$$
	H_L(x,z,t/c_0) \leq H_L(x,x,t/2) \exp\left( c_0^\ast C\frac{\dist(x,z)^2}{t} + \frac{c_0}{2}C \right)
$$
for every $z\in M$, with $c_0^\ast = \frac{2c_0}{c_0 - 2} = \left( \frac{1}{2} - \frac{1}{c_0} \right)^{-1}$. Integrating on $B_{\sqrt{t}}(x)$ we get
\begin{equation} \label{HL_in2}
	\int_{B_{\sqrt{t}}(x)} H_L(x,z,t/c_0) \, \di z \leq e^{(c_0^\ast + c_0/2) C} |B_{\sqrt{t}}(x)| H_L(x,x,t/2) \, .
\end{equation}
Putting together \eqref{HL_in1} and \eqref{HL_in2} we obtain
\begin{equation} \label{G_lower0}
	H_L(x,y,t) \geq \frac{e^{-(2+c_0/2+c_0^\ast)C}}{|B_{\sqrt{t}}(x)|} \exp\left( -2C\frac{\dist(x,y)^2}{t} \right) \int_{B_{\sqrt{t}}(x)} H_L(x,z,t/c_0) \, \di z \, .
\end{equation}
From the upper bound \eqref{G_upper} and the co-area formula we have
\begin{align*}
	\int_{M\setminus B_{\sqrt{t}}(x)} H_L(x,z,t/c_0) \, \di z & \leq C_3' \int_{\sqrt{t}}^{\infty} \frac{|\partial B_r(x)|}{\left|B_{\sqrt{t/c_0}}(x)\right|} \exp\left(-c_0C_4'\frac{r^2}{t}\right) \, \di r \\
	& = C_3' \int_{\sqrt{c_0}}^{\infty} \frac{\sqrt{t/c_0} \left|\partial B_{s\sqrt{t/c_0}}(x)\right|}{\left|B_{\sqrt{t/c_0}}(x)\right|} e^{-C_4's^2} \, \di s
\end{align*}
where we have changed variable $s = r\sqrt{c_0/t}$. Since $\Ricc\geq0$ we have
$$
	\frac{\sqrt{t/c_0} \left|\partial B_{s\sqrt{t/c_0}}(x)\right|}{\left|B_{\sqrt{t/c_0}}(x)\right|} \leq s^{m-1} \frac{\sqrt{t/c_0} \left|\partial B_{\sqrt{t/c_0}}(x)\right|}{\left|B_{\sqrt{t/c_0}}(x)\right|} \leq m s^{m-1} \, ,
$$
where the first inequality follows by Bishop-Gromov theorem and the second from the inequality $R|\partial B_R(x)| \leq m|B_R(x)|$, holding for every $R>0$ and for any base point $x$ on a Riemannian manifold with $\Ricc\geq0$, see for instance \cite[Formula (19)]{liann}. Substituting in the above estimate and recalling \eqref{c_0} and Lemma \ref{lem_stoch} we get
$$
	\int_{B_{\sqrt{t}}(x)} H_L(x,z,t/c_0) \, \di z = 1 - \int_{M\setminus B_{\sqrt{t}}(x)} H_L(x,z,t/c_0) \, \di z \geq 1 - \gamma > 0
$$
and from \eqref{G_lower0} we obtain
$$
	H_L(x,y,t) \geq \frac{C_1'}{|B_{\sqrt{t}}(x)|} \exp\left( -2C\frac{\dist(x,y)^2}{t} \right)
$$
where $C_1' = (1-\gamma)e^{-(2+c_0/2+c_0^\ast)C} > 0$ only depends on $m$ and $\alpha$. This proves \eqref{G_lower}.

\end{appendix}

\begin{center}
{\bf Acknowledgement}
\end{center} 
This work was done when E.S.G. was a professor at Universidade Federal Rural do Semi-\'Arido, Cara\'ubas, and he thanks the institution for the fruitful working environment.

\bibliographystyle{plain}

\end{document}